\documentclass[12pt]{amsart}
\usepackage{amsmath,amsfonts,amsthm,amssymb,mathrsfs,amsxtra,amscd,latexsym, xcolor}
\usepackage[hmargin=2cm,vmargin=3cm]{geometry}
\usepackage{hyperref}

\usepackage{amsmath,amsthm,amssymb,enumerate,comment,amsfonts,mathrsfs,amsxtra,amscd,latexsym}

\usepackage{times}
\usepackage[T1]{fontenc}
\usepackage{graphics}
\usepackage{epsfig}
\usepackage{accents}
\usepackage{float}
\usepackage{thmtools}
\usepackage{thm-restate}
\usepackage{hyperref}
\usepackage{cleveref}
\usepackage{float}

\input amssym.def
\input amssym.tex
\usepackage{color}
\usepackage{url}
\usepackage{mathtools}
\usepackage{multirow}
\DeclareSymbolFont{bbold}{U}{bbold}{m}{n}
\DeclareSymbolFontAlphabet{\mathbbold}{bbold}

\newtheorem{theorem}{Theorem}[section]

\newtheorem{lemma}[theorem]{Lemma}

\newtheorem{definition}[theorem]{Definition}

\newtheorem{remark}{Remark}
\newtheorem{que}{Question}
\newtheorem{problem}{Problem}

\numberwithin{equation}{section}

\newcommand{\sm}{\left(\begin{smallmatrix}}
\newcommand{\esm}{\end{smallmatrix}\right)}
\newcommand{\mat}{\left(\begin{matrix}}
\newcommand{\emat}{\end{matrix}\right)}

\def\tr{\mathrm{Tr}}

\def\SL{\mathrm{SL}}

\def\Q{\mathbb Q} \def\R{\mathbb R} \def\Z{\mathbb Z} \def\C{\mathbb C}
\def\N{\mathbb N} \def\M{\mathbb M}  \def\H{\mathbb H}
 \def\T{\mathbb T} \def\P{\mathbb P}
\def\t{\tau}

\newcommand{\beq}{\begin{eqnarray*}}
\newcommand{\eeq}{\end{eqnarray*}}
\newcommand{\beqn}{\begin{eqnarray}}
\newcommand{\eeqn}{\end{eqnarray}}

\newcommand{\ben}{\begin{enumerate}}
\newcommand{\een}{\end{enumerate}}

\newcommand{\eps}{\epsilon}

\newcommand{\Mod}[1]{\ \mathrm{mod}\ #1}

\usepackage{listings} 
\usepackage{xspace} 
\lstdefinelanguage{Sage}[]{Python}
{morekeywords={True,False,sage,singular},
sensitive=true}
\definecolor{dblackcolor}{rgb}{0.0,0.0,0.0}
\definecolor{dbluecolor}{rgb}{.01,.02,0.7}
\definecolor{dredcolor}{rgb}{0.8,0,0}
\definecolor{dgraycolor}{rgb}{0.30,0.3,0.30}

\DeclareMathOperator{\Tr}{Tr}

\newcommand{\F}{\mathbb{F}}

\renewcommand{\H}{\mathbb{H}}

\renewcommand{\hat}{\widehat}

\newcommand{\pfrac}[2]{\left(\frac{#1}{#2}\right)}

\newcommand{\tth}{^{\operatorname{th}}}

\DeclareMathOperator{\ord}{ord}
\DeclareMathOperator{\qord}{ordq}
\DeclareMathOperator{\Gal}{Gal}
\DeclareMathOperator{\Conj}{Conj}
\DeclareMathOperator{\lcm}{lcm}
\DeclareMathOperator{\mult}{mult}
\DeclarePairedDelimiter{\ceil}{\lceil}{\rceil}
\DeclarePairedDelimiter{\ordG}{|}{|}
\newcommand{\bbone}{\mathbbold{1}}
\renewcommand{\exp}[1]{\mathrm{exp}\left(#1\right)}
\renewcommand{\d}{d}
\newcommand{\num}{c}
\newcommand{\noin}{\noindent}
\newcommand{\bigpmod}[1]{\hspace{3ex}\left(\mathrm{mod }\; \;#1\right)}
\renewcommand{\N}{\Z_{>0}}

\begin{document}

\title{Moonshine for all finite groups}

\author{Samuel DeHority}
\address{Department of Mathematics, UNC-Chapel Hill, Chapel Hill, NC 27514}
\email{sdehorit@live.unc.edu}

\author{Xavier Gonzalez}
\address{Department of Mathematics, Harvard  University, Cambridge, MA 02138}
\email{haroldxaviergonzalez@college.harvard.edu}

\author{Neekon Vafa}
\address{Department of Mathematics, Harvard  University, Cambridge, MA 02138}
\email{nvafa@college.harvard.edu}

\author{Roger Van Peski}
\address{Department of Mathematics, Princeton University, Princeton, NJ 08544}
\email{rpeski@princeton.edu}

\begin{abstract}  In recent literature, moonshine has been explored for some groups beyond the Monster, for example the sporadic O'Nan and Thompson groups. This collection of examples may suggest that moonshine is a rare phenomenon, but a fundamental and largely unexplored question is how general the correspondence is between modular forms and finite groups. 
For every finite group $G$, we give constructions of infinitely many graded infinite-dimensional $\C[G]$-modules where the McKay-Thompson series for a conjugacy class $[g]$ is a weakly holomorphic modular function properly on $\Gamma_0(\ord(g))$. As there are only finitely many normalized Hauptmoduln, groups whose McKay-Thompson series are normalized Hauptmoduln are rare, but not as rare as one might naively expect. We give bounds on the powers of primes dividing the order of groups which have normalized Hauptmoduln of level $\ord(g)$ as the graded trace functions for any conjugacy class $[g]$, and completely classify the finite abelian groups with this property. In particular, these include $(\Z / 5 \Z)^5$ and $(\Z / 7 \Z)^4$, which are not subgroups of the Monster.
\end{abstract}

\thanks{\emph{Keywords:} Moonshine, Modular forms}

\thanks{The authors would like to thank Ken Ono and John Duncan for advising this project and for their many helpful conversations and suggestions. We also thank Hannah Larson for helpful conversations and edits, and Robert Wilson for answering our question about subgroups of the Monster. Finally, we thank Emory University, Princeton University, and the NSF (via grant DMS-1557690) for their support.
}

\thanks{2010
 Mathematics Subject Classification: 11F11, 11F30, 11F33, 20C05}

\maketitle

\section{Introduction and statement of results} \label{sec:intro}

\noin The theory of moonshine began with the single distinguished example of \emph{monstrous moonshine}. McKay and Thompson \cite{ ThoFinite, Tho} observed that the Fourier coefficients of the \emph{normalized elliptic modular invariant} 
\begin{equation*}
J(\tau) =  q^{-1} + 196884q + 21493760q^2+ \cdots  ~\text{where }  \tau \in \H,~ q := \exp{2\pi i \tau} 
\end{equation*}
are naturally given by nontrivial positive linear combinations of the 194 dimensions of irreducible representations of the monster group\footnote{At this stage the existence of $\M$ was still entirely conjectural.} $\M$. These early observations included
$$1 = 1, \hspace{1 cm} 196884 = 196883+1, \hspace{1 cm} 21493760 = 1 + 196883 + 21296876,$$
where the summands are the dimensions of the three smallest irreducible representations of $\M$. 
This observation led Thompson \cite{Tho} to conjecture the existence of an infinite-dimensional graded module, notated now as
\begin{equation}
V^\natural = \bigoplus_{n=-1}^\infty V^\natural_n
\end{equation}
in which the graded traces, now called \emph{McKay-Thompson series}, 
\begin{equation}
T_g(\tau) := \sum_{n=-1}^\infty \tr(g | V^\natural_n)q^n
\end{equation}
would be modular functions of interest strictly on level dividing $h \ord(g)$ for some $h | \gcd(12, \ord(g))$, and in particular in the case of the identity $T_e(\tau) = J(\tau)$. Here and throughout the paper, we say that a function is \emph{strictly} on level $N$ if it is invariant under $\Gamma_0(N)$ and not any proper divisor of $N$. Note that this requirement rules out the use of a trivial module made only from copies of the trivial representation, as the graded-trace functions on conjugacy classes of different orders give rise to different modular functions. Moreover, we note that $V^\natural$ has recently been discovered \cite{moonshine_survey} to be asymptotically equal to copies of the regular representation, meaning that $$ \lim_{n \to \infty} \frac{\mult_{i}(n)}{\sum_{j = 1}^{194} \mult_{j}(n)} = \frac{\dim{\chi_{i}}}{\sum_{j = 1}^{194} \dim{\chi_{j}}},$$
where $\mult_{i}(n)$ gives the number of copies of the $i \tth$ irreducible representation in $V^\natural_n$. In particular, the proportion of the trivial representation $\chi_1$ tends to 
$$ \lim_{n \to \infty} \frac{\mult_1(n)}{\sum_{j = 1}^{194} \mult_j(n) } = \frac{1}{5844076785304502808013602136}.$$

Building on Thompson's prediction, Conway-Norton \cite{Conway-Norton} conjectured the precise candidates for the functions $T_g$. For each $g \in \M$, there exists a group $\Gamma_g < \SL_2(\R)$ commensurable with $\SL_2(\Z)$, such that $\Gamma_g \backslash \H$ is genus zero and each McKay-Thompson series $T_g(\tau)$ is the \emph{normalized Hauptmodul} for $\Gamma_g$. The normalized Hauptmodul $T_g(\tau)$ is the unique function whose Fourier expansion begins $q^{-1}+O(q)$ and that generates the field of meromorphic functions on the compactification of $\Gamma_g \backslash \H$. Throughout this paper, when we write \emph{the} Hauptmodul for a given genus-zero group, we refer to the normalized Hauptmodul. These $\Gamma_g$ are the congruence subgroups $\Gamma_{0}(h\ord(g))$ usually extended by Atkin-Lehner and other involutions. Atkin-Fong-Smith \cite{AFS} showed the existence of a possibly virtual graded $\M$-module with the required properties, where by a \emph{virtual module} we mean a formal sum of irreducible representations with possibly negative coefficients. Frenkel-Lepowsky-Meurman \cite{FLM1, FLM2, FLM3} explicitly constructed the module, with additional algebraic properties, that they conjectured would satisfy Conway and Norton's prediction. Finally, in 1992, Borcherds \cite{Bor1, Bor2} showed that the construction did satisfy the full Conway-Norton conjecture. In a tour de force, Borcherds provided a refined algebraic description of the module from which he was able to relate the denominator formulas of the module to the replication formulas required of the Hauptmoduln (see Section \ref{ssc:replicability} for a discussion of replicability).

These results are naturally formulated in terms of quantum field theory. Here the $q$-expansion of $J(\tau)$ is the partition function for a physical system with monster symmetry whose states are given by $V^\natural$; see \cite{duncan_rademacher_sums_gravity} section 7 and the references cited therein. 
Recent work has expanded the horizon of moonshine, discovering modular forms arising as graded traces of infinite-dimensional modules of other finite groups. These finite groups include subquotients of the Monster \cite{carnahan_gen_moonshine_IV, Conway-Norton, mason, norton_generalized, queen} like $M_{24}$ \cite{EOT, muchadoaboutmatthieu}, which is the first example of \emph{umbral moonshine} \cite{CDH1, CDH2, umbral_moonshine}, a theory that relates the Niemeier lattices to vector-valued mock modular forms. Moonshine has even been extended to the O'Nan group \cite{onan} which is a pariah group, i.e. a sporadic simple group that is not a subquotient of $\M.$

The McKay-Thompson series of the moonshine modules for these finite groups are distinguished: for example, the mock modular forms arising in umbral moonshine have minimal principal parts, analogous to the requirement in monstrous moonshine that the graded-trace functions be Hauptmoduln. In view of these results, it is natural to ask about the ubiquity of moonshine-like phenomena among all finite groups if we allow for more flexibility for the McKay-Thompson series. 
In fact, building on the results of Zhu \cite{zhu}, Dong-Li-Mason (see Theorem 3 in \cite{DLM}) show that if a vertex operator algebra satisfies some properties, then the graded trace functions for any of its symmetry groups are ``generalized modular forms,'' essentially meaning that these functions and their transforms have $q$-expansions with fractional powers in $q$ and bounded denominators, and that there are finitely many transforms under certain modular groups up to scalar multiplication. However, in analogy with monstrous moonshine, we would like the graded trace functions to be precisely modular forms. This paper shows that between modular forms and the representation theory of finite groups, there are more prevalent relations than those coming from $\M$ and related groups. Specifically, we address two related but in some sense opposite problems:

\begin{problem}
Given a group $G$, does there exist a graded infinite-dimensional $\C[G]$-module $V^G$ whose graded traces are well-behaved modular functions strictly on level $\ord(g)$, similarly to the case of $V^\natural$? \footnote{We note that the levels of modular functions in $V^\natural$ are technically $h\ord(g)$ where $h|(12,\ord(g))$. Our proof of Theorem \ref{thm:moonshine_always} may be easily altered to handle other levels, so for simplicity we only prove it for the case when the graded trace of $g$ is strictly on level $\ord(g)$.}
\end{problem}

\begin{problem}
Given a collection of Hauptmoduln arising in monstrous moonshine, for which groups $G$ does there exist a $V^G$ with $\Tr(g | V^G)$ equal to the Hauptmodul for $\Gamma_{0}(\ord(g))$? (In general $\Gamma_0(\ord(g))$ need not be genus-zero, so this puts severe restrictions on the orders of group elements).
\end{problem}

For the first question, we will show that relaxing the requirement for the $T_g$ to be Hauptmoduln, a condition which is roughly equivalent to them satisfying the \emph{replicability} condition defined in Section \ref{ssc:replicability}, allows construction of infinite dimensional graded modules analogous to $V^\natural$ for every finite group. Further, without relaxing these requirements it is still possible to produce moonshine modules for groups which are not subgroups of the monster.

We say that a finite group $G$ has \textbf{moonshine of depth $\d \ge 1$} if the following hold:
\begin{enumerate}\label{def:depth}
    \item[(i)] There exists a graded infinite-dimensional $\C[G]$-module \begin{equation*}
        V^G= \bigoplus_{n \in \{-\d\} \cup \N} V_n^G,
    \end{equation*}
    \item[(ii)] The McKay-Thompson series $T_g(\tau)$ on $V^G$ is a weakly holomorphic modular function strictly on $\Gamma_0(\ord(g))$, and
    \item[(iii)] We have $T_e(\tau) = J(\tau) \mid dT(d)$, where $\mid dT(d)$ denotes the action of the normalized $d \tth$ Hecke operator as defined in \eqref{eq:hecke}.
\end{enumerate}
We refer to the smallest $\d$ for which such a module exists as the \textbf{depth} of $G$. The first main result is that every finite group has moonshine of some finite depth.

\begin{remark}
We note that in this paper we consider only the modular curves $\Gamma_{0}(N) \backslash \H$. We do not consider the modular curves generated by extending congruence subgroups with Atkin-Lehner involutions. Doing so would be a natural next step; see Question \ref{qu:atkin_lehner} in Section \ref{ssc:discussion}.
\end{remark}

\begin{theorem}\label{thm:moonshine_always}
Let $G$ be a finite group. For infinitely many positive integers $\d$, there exists an infinite-dimensional graded $\C[G]$-module $$V^G = \bigoplus_{n \in \{-\d\} \cup \N} V_n^G$$
such that for all $g \in G$, the McKay-Thompson series $\Tr(g | V^G)$ is a modular function strictly of level $\ord(g)$, and $\Tr(e | V^G) = J(\tau) \mid d T(d)$.
\end{theorem}

\begin{remark}\label{rmk:tensor_power_of_V_natural}
The choice of graded trace functions given in the proof of Theorem \ref{thm:moonshine_always} were quite arbitrarily chosen, primarily for the ease of understanding the asymptotics of their coefficients. There are many other choices that could have been made. In fact, embed $G$ into $S_n$ by Cayley's theorem and consider $S_n$ acting on $(V^\natural)^{\otimes n}$ by permuting the components. Then the graded trace of an $n$-cycle is $T_g(\tau) = J(n\tau)$. But since the graded trace on a tensor product $T_{V\otimes W}(\tau)$ is $T_V(\tau)T_W(\tau)$, if $g$ has cycle type $(k_1, \ldots, k_d)$ then
$$T_g(\tau) = J(k_1\tau)J(k_2\tau)\cdots J(k_d\tau),$$ which is strictly level $\ord(g)$. While this module has the structure of a vertex operator algebra and we do not provide one for the construction in the proof of the theorem, we could have replaced $V^\natural$ in this construction with any graded vector space whose graded dimension is $f \in \mathbb{Z} (( q))$. This construction and the one in the proof show that there are many ways to provide graded $\C[G]$-modules with specified graded trace functions. Although the notion of depth used in this paper was defined for the case where $\Tr(e | V^G) =  J(\tau) \mid d T(d)$, one could alternatively define depth as the order of the pole at infinity, and then this construction would give a bound $d \leq \ordG{G}$.
\end{remark}

We note that the representation will be asymptotically regular if and only if the coefficients of modular function $\Tr(e | V^G)$ are asymptotically larger than coefficients of the graded trace functions associated to the other conjugacy classes. On the other hand, if $$ \lim_{n \to \infty} \frac{\mult_{1}(n)}{\sum_{j = 1}^{k} \mult_{j}(n)} = 1,$$ where $k = |\Conj(G)|$ and $\mult_{1}(n)$ stands for the multiplicity of the trivial representation in $V_n^G$, we call the representation asymptotically trivial.

\begin{remark}\label{rmk:repr_dists}
An inspection of the proof of Theorem \ref{thm:moonshine_always} in Section \ref{ssc:prf_1.1} shows that the modules constructed are asymptotically trivial. However, by relaxing the requirement that  $R_e = J(\tau) \mid d T(d)$ and increasing the order of the pole at infinity of $R_e$, we may also get a module that is asymptotically equal to the regular representation, as demonstrated in Section \ref{ssc:ex_s4} where $G = S_4$. A natural question is what other asymptotic tendencies are possible in this spectrum of moonshine modules.
\end{remark}
\begin{remark}
We note that we could have defined the depth of moonshine in item (iii) above differently to accommodate arbitrary principal parts instead of merely $q$-expansions of the form $q^{-d} + O(q)$. In fact, this modified notion of depth better accommodates the modules we construct that are asymptotic to the regular representation, as in Remark \ref{rmk:repr_dists}. 
\end{remark}
\noin Note that depth one moonshine includes monstrous moonshine if we allow the graded trace function for $g \in G$ to have level $h\ord(g)$ for $h | (12, \ord(g))$. The graded trace functions of monstrous moonshine, however, are always Hauptmoduln of genus zero groups. As there are only finitely many of these Hauptmoduln, groups whose graded trace functions are Hauptmoduln are rare.
In fact, given a set of Hauptmoduln $\{T_n\}$ with each $T_n$ of level $n$, it is possible determine whether a finite group $G$ has an infinite dimensional graded module with McKay-Thompson series $T_g = T_{\ord(g)}$, the Hauptmoduln on the genus-zero congruence subgroups $\Gamma_0(N) \leq \SL_2(\Z)$. It is well known that $\Gamma_0(N)$ is genus-zero if and only if
\begin{equation}\label{eq:genus_zero_N}
N \in \{1,2,3,4,5,6,7,8,9,10,12,13,16,18,25\}.
\end{equation}
We obtain necessary bounds on the orders of primes dividing $\ordG{G}$ by computationally checking congruences among the Hauptmoduln of the $15$ genus-zero congruence subgroups listed above, yielding the following result.

\begin{theorem}\label{thm:order_bounds}
Let $G$ be a finite group with a $\C[G]$-module
$$V^G = \bigoplus_{n \in \{-1\} \cup \N} V_n^G$$
for which $\Tr(g | V^G) =: T_{\ord(g)}$ is the Hauptmodul of $\Gamma_0(\ord(g))$ for all $g \in G$. Then $\ordG{G}$ is only divisible by primes $p$ with $(p-1)|24$, and furthermore $\ord_p(\ordG{G}) \leq \frac{24}{p-1}$ for $p$ odd and $\ord_2(\ordG{G}) \leq 25$. 
\end{theorem}

\noin For abelian groups, we obtain an explicit description of the generating function of the multiplicity of an irreducible representation in each homogeneous component of $V^G$ in Lemmas \ref{lem:raw_congs} and \ref{lem:nice_congs}. These are of independent interest and also yield concrete conditions on when such multiplicities are positive integers, from which we obtain a complete classification of all abelian groups having moonshine in the sense of Theorem \ref{thm:order_bounds}.

\begin{restatable}{theorem}{abelianmoonshine}\label{thm:abelian_moonshine}
Let $G$ be a finite abelian group. Then there exists a $\C[G]$-module $V^G = \bigoplus_{n \in \{-1\} \cup \N} V_n^G$ for which $\Tr(g | V^G) = T_{\ord(g)}$ is the Hauptmodul of $\Gamma_0(\ord(g))$ if and only if $G$ is isomorphic to one listed in Table \ref{table:abelian_groups}. Furthermore $V_n^G$ is asymptotically regular.
\end{restatable}
\begin{table}[H] \label{ab_table}
\begin{center}
  \begin{tabular}{| c  | r| }
  \hline
    Group & Conditions on exponents \\
    \hline
    \multirow{5}{*}{$(\Z/2\Z)^a \times (\Z/4\Z)^b \times (\Z/8\Z)^c \times (\Z/16\Z)^d$} & $a+2b+3c+4d \leq 16$ \\
     & $b+2c+3d \leq 8 $ \\
     & $c+2d \leq 4 $\\
     & $d \leq 2 $ \\
     & $a+b+c+d \leq 13$ \\
    \hline
    \multirow{2}{*}{$(\Z / 2\Z)^a \times (\Z / 3\Z)^b$} & $1 \leq a \leq 4$ \\
    & $1 \leq b \leq 3$ \\
    \hline
    \multirow{3}{*}{$(\Z / 2\Z)^a \times (\Z / 4\Z)^b \times (\Z / 3\Z)^c$} & $a + 2b \leq 4$ \\
    & $1 \leq b \leq 2$ \\
    & $1 \leq c \leq 2$ \\
    \hline
    \multirow{1}{*}{$(\Z / 2\Z) \times (\Z / 3\Z)^a \times (\Z / 9\Z)$} & $a \leq 1$ \\
    \hline
    \multirow{2}{*}{$(\Z / 3\Z)^a \times (\Z / 9 \Z)^b$} & $a + 2b \leq 9$\\
    & $b \leq 3$ \\
    \hline 
    \multirow{2}{*}{$(\Z / 2\Z)^a \times (\Z / 5\Z)^b$} & $1 \leq a \leq 3$ \\
    & $1 \leq b \leq 2$ \\
    \hline
    \multirow{2}{*}{$(\Z / 5\Z)^a \times (\Z / 25 \Z)^b$} & $ a + 2b \leq 5$ \\
    & $b \leq 1$\\
    \hline
    $(\Z / 7\Z)^a$ & $a \leq 4$\\
    \hline
    $(\Z / 13\Z)^a$ & $a \leq 2$ \\
    \hline
  \end{tabular}
  \vspace{2ex}
  \caption{\label{table:abelian_groups} Abelian groups with modules where the $T_g$ are Hauptmoduln. $a,b,c,d\ge 0$.}
\end{center}
\end{table}

\begin{remark}\label{rmk:not_subgroups}
For those groups in Table \ref{ab_table} which are subgroups of $\M$, the existence of the module $V^G$ is immediate by simply restricting $V^\natural$. We note that some of these groups, for example $(\Z/7\Z)^4$ and $(\Z/5\Z)^5$, are not subgroups of $\M$ (\cite{wilson_communication}; see also Theorems 5 and 7 of \cite{monster_subgroup}) but still have an associated moonshine module. It is interesting to note, however, that$(\Z/5\Z)^5$ is a subgroup of the Weyl group $W(A_4^6)$ and $(\Z/7\Z)^4$ is a subgroup of Weyl group $W(A_6^4)$, which plays a role in umbral moonshine.
\end{remark}
\begin{remark}
One can extend Theorem \ref{thm:abelian_moonshine} to show that if we choose a finite set $A \subset \N$ where for all $a \in A$, $f_a$ is a weakly holomorphic modular function strictly of level $a$ where the $f_a$ are linearly independent, then there exist finitely many finite abelian groups $G$ having a graded $\C[G]$-module where the graded trace function for $g$ is $f_{\ord(g)}$.
\end{remark}

\noin Having introduced our main results, we now provide an outline of the paper and a description of our methods. Section $2$ gives an overview of background material on modular forms and asymptotics of their Fourier coefficients. Section $3$ shows that virtual moonshine is possible for any finite group because there is an ample supply for modular forms which can be constructed to satisfy certain congruences that guarantee integral multiplicities. These congruences come from an application of the Schur orthogonalilty relations to the group's irreducible characters and the proposed McKay-Thompson series. Furthermore, we show in Section $4$ that the multiplicities can in fact be chosen to be positive, proving Theorem \ref{thm:moonshine_always}. This gives an answer to Problem $1$, while Theorems \ref{thm:order_bounds} and \ref{thm:abelian_moonshine} answer an interesting case of Problem $2$. Building on the virtual moonshine shown in Section $3$, we construct explicit functions which give rise to moonshine modules for every finite group (Theorem \ref{thm:moonshine_always}). We then prove Theorems \ref{thm:order_bounds} and \ref{thm:abelian_moonshine} by reducing to a computer-check of congruences among Hauptmodul Fourier coefficients, as well as calling upon asymptotics to show nonnegativity in Theorem \ref{thm:abelian_moonshine}. Section $5$ provides two worked examples. First we construct a moonshine module for $(\Z / 7\Z)^4$ where all the graded-trace functions are Hauptmoduln. This demonstrates our proof of Theorem \ref{thm:abelian_moonshine} by showing how the congruence and nonnegativity conditions are determined and satisfied. Then we construct a moonshine module of depth $791$ for $S_4$ to demonstrate the proof of Theorem \ref{thm:moonshine_always} by constructing graded-trace functions for a module with only finitely many virtual components and then applying Hecke operators to these initial graded-trace functions to get new graded-trace functions that ensure nonnegativity of multiplicity. In this example, as in Remark \ref{rmk:repr_dists}, we convert the resulting asymptotically trivial representation to an asymptotically regular one.

\section{Background}\label{sec:nuts_bolts}

\noin In Section \ref{ssc:mod_form}, we give necessary background for modular forms. In Section \ref{ssc:asymp}, we give exact formulas for coefficients and asymptotics for specific modular forms. We will use these asymptotics to show the nonnegativity of the multiplicities of the irreducible representations in the moonshine modules  we construct in Section \ref{sec:pf_thms}. In Section \ref{ssc:replicability}, we describe the notion of replicability, which helps to illuminate what makes monstrous moonshine so special.

\subsection{Modular Forms and Functions}\label{ssc:mod_form}

We recall the basic theory of modular forms, as detailed e.g. in \cite{webofmodularity}. The group $\SL_2(\Z)$ acts on the upper half-plane $\H$ by the map
$ A \t \mapsto \frac{a \t + b}{c \t + d}$, and a \emph{cusp} of a subgroup $\Gamma \leq \SL_2(\Z)$ is an equivalence class of $\Q \P^1 = \Q \cup \{\infty\}$ under the action of $\Gamma$.

\begin{definition}\label{def:mod_form}
A function $f$ is a \textbf{holomorphic modular form of integral weight $k$} $\geq 0$ on a subgroup $\Gamma \leq \SL_2(\Z)$ if it satisfies the following conditions:
\begin{itemize}
    \item[(i)] $f$ is holomorphic on $\H$,
    \item[(ii)] We have
    $$f(A \t) = (c\t + d)^kf(\t)$$
    for all $A = \begin{pmatrix}
a & b\\
c & d
\end{pmatrix} \in \Gamma, \t \in \H$, and
    \item[(iii)] $f$ is bounded as $\t$ approaches all the cusps of $\Gamma$.
\end{itemize}
If property (iii) is not satisfied and $f$ has at most exponential growth at cusps, then we say $f$ is a \textbf{weakly holomorphic modular form}. Furthermore, if $k = 0$, then we say $f$ is a \textbf{modular function}. 
\end{definition} 

\noin All groups $\Gamma \leq \SL_2(\Z)$ considered in this paper have \emph{width $1$ at infinity}, i.e. contain $T:= \begin{pmatrix} 1 & 1 \\ 0 & 1 \end{pmatrix}$. Hence any modular form on $\Gamma$ is invariant under $\tau \mapsto \tau+1$, so it has a Fourier expansion. 

\begin{equation*}
f = \sum_{n \gg -\infty}^\infty a_n q^n,
\end{equation*}
where here and throughout the paper we define $q := \exp{2 \pi i \t}$.
\begin{definition}\label{def:cong_sub}
For $N \in \N$, we define the congruence subgroup 
$$\Gamma_0(N) := \left\{\begin{pmatrix}
a & b\\
c & d
\end{pmatrix} \in \SL_2(\Z): c \equiv 0 \pmod{N}\right\}.$$
Note that $\Gamma_0(1) = \SL_2(\Z)$.
\end{definition}
\noin We say that a modular form $f$ is \emph{strictly of level $N$} if $f$ is a modular form on $\Gamma_0(N)$ and \emph{not} on $\Gamma_0(d)$ for any proper divisor $d \mid N$. Denote by $M_k(N)$ (resp. $M^!_k(N)$) the space of entire (resp. weakly holomorphic) modular forms strictly of level $N$ and weight $k$.
\\\\
Recall that the normalized Eisenstein series of weight $k$ is defined by
$$E_{k}(\t) := 1 - \frac{2k}{B_k}\sum_{n=1}^\infty \sigma_{k-1}(n)q^n$$
where
$$\sigma_k(n) := \sum_{d | n} d^k$$
and $B_k$ is the $k$th Bernoulli number. For even $k \geq 4$, it is well-known that
$$E_{k}(\t) = \frac{1}{2\zeta(k)} \sum_{(m,n) \in \Z^2 \setminus \{(0,0)\}} \frac{1}{(m \t + n)^k}$$
and hence $E_{k} \in M_k(1)$ when $k \geq 4$ is even. Also recall the Dedekind eta-function, defined by
$$\eta(\t) := q^{1/24} \prod_{n=1}^\infty (1-q^n)$$
for which
\begin{equation}\label{eq:eta_S}
\eta(\t + 1) = \exp{\pi i / 12} \eta(\t), \hspace{1.5em} \eta(-1/\t) = \sqrt{- i \t} \cdot \eta(\t)
\end{equation}
hold. Recall the modular discriminant defined by
\begin{align*}
\Delta(\t) &:= \eta(\t)^{24}
=\frac{E_4(\t)^3 - E_6(\t)^2}{1728} = q - 24q^2 + O(q^3) \in M_{12}(1)
\end{align*}
is a cusp form. It is well known that $\Delta(N \tau)$ is strictly of level $N$. We define the normalized $j$-function,
\begin{equation*}
J(\t) := \frac{E_4(\t)^3}{\Delta(\t)} - 744 = q^{-1} + 196884q + O(q^2) \in M_0^!(1).
\end{equation*}
\noin The action of the $m^{\text{th}}$ order normalized \emph{Hecke operator} on a weight $0$ modular form is given by 
\begin{equation}\label{eq:hecke}
f(\tau) \mid mT(m) := \sum_{\substack{a d = m 
\\ 0 \leq b < d}} f\left(\frac{a \t + b}{d}\right).
\end{equation}
In particular, for primes $p$,
\begin{equation}\label{eq:hecke_prime}
f(\tau) \mid pT(p) = f(p \t) + \sum_{b=0}^{p-1} f\left(\frac{\tau + b}{p}\right).
\end{equation}
\\
The integrality of certain $\C[G]$-modules rests on congruences between Fourier coefficients of modular forms, for which the following notation is convenient. For a modular form
$$f = \sum_{n \gg -\infty} a(n)q^n$$
where $a(n) \in \Z$, define
$$\qord_m(f) := \inf \{n : m \nmid a(n) \}$$
with the convention that $\qord_m(f) = \infty$ if $m | a(n)$ for all $n$. Then we have the following result from \cite{sturm}, which allows finite computations to prove congruences between modular forms:
\begin{theorem}[Sturm]\label{thm:sturm}
Let $f \in M_k(N)$ with integer coefficients. If
$$\qord_p(f) > \frac{k [\SL_2(\Z) : \Gamma_0(N)]}{12},$$
then
$$\qord_p(f) = \infty$$
for all primes $p$.
\end{theorem}
\noin From this theorem, we obtain the following lemma, allowing us to apply the Sturm bound to weakly holomorphic modular functions:
\begin{lemma}\label{lem:sturm} 
Let $f \in M^!_0(N)$ with integer coefficients, and let $s$ denote the order of the highest order pole of $f$ at any cusp. If
$$\qord_m(f) > s ([\SL_2(\Z) : \Gamma_0(N)] - 1),$$
then
$$\qord_m(f) = \infty$$
for all $m \geq 2$.
\end{lemma}
\begin{proof}
By the Chinese remainder theorem it suffices to prove the result for $m=p^n$ a power of a prime. Since the highest order pole of $f$ is $s$, we know that
$$f \Delta^s \in M_{12s}(N).$$
Since $\qord_p(gh) = \qord_p(g) + \qord_p(h)$ for all $q$-series $g$ and $h$, and since $\qord_p(\Delta^s) = s$, we have
$$\qord_p(f \Delta^s) = s + \qord_p(f).$$
From this equality and the hypothesis on $\qord_p(f)$, we immediately get
$$
\qord_p(f \Delta^s) > s[\SL_2(\Z) : \Gamma_0(N)]
$$
so by Theorem \ref{thm:sturm}, $\qord_p(f \Delta^s) = \infty$ and hence $\qord_p(f) = \infty$ as well. Thus $f/p$ has integer coefficients so we may make the same argument, and arguing inductively we have that $\qord_{p^n}(f) = \infty$.
\end{proof}

\subsection{Modular Form Asymptotics}\label{ssc:asymp}

First, we recall exact formulas and asymptotics for the coefficients of Hauptmoduln and images of $J$ under Hecke operators in Lemmas \ref{lem:hecke_coefs} and \ref{lem:haupt_coeefs}, as found in \cite{moonshine_survey, hannah}. To do so, we first recall two objects necessary for these formulas: the \emph{modified Bessel function of the first kind} $$ I_{1}(x) := \frac{x}{2} \sum_{k = 0}^{\infty} \frac{(\frac{1}{4} x^{2})^{k}}{k! \, \Gamma(k + 2)},$$ and the \emph{classical Kloosterman sum} 
\begin{equation} \label{eq:kloosterman_def}
K(m,n,c) :=  \sum_{\substack{x \pmod{c} \\ (x,c) = 1}} \exp{\frac{2 \pi i (m x + n \overline{x})}{c}},
\end{equation}
where $x \cdot \overline{x} \equiv 1 \pmod{c}$.
\begin{lemma}\label{lem:hecke_coefs} If $m$ is a positive integer, then
\begin{equation*}
J(\tau) \mid mT(m) = q^{-m} + \sum_{n \geq 1} c_{e}(m,n) q^{n}
\end{equation*}
where
\begin{equation*}\label{eq:exact_j_hecke}
c_{e}(m,n) = 2 \pi \sqrt{\frac{m}{n}} \sum_{c > 0} \frac{K(-m,n,c)}{c} \cdot I_{1}\left(\frac{4 \pi \sqrt{m n}}{c}\right).
\end{equation*}
\end{lemma}
\noin Moreover, by Theorem $1.2$ of \cite{hannah}, we have that, as $n \to \infty$,
\begin{equation}\label{eq:j_asym}
c_{e}(m,n) \sim \frac{m^{1/4}}{\sqrt{2}n^{3/4}} \exp{4\pi \sqrt{mn}}.
\end{equation}
\begin{remark}
Lemma \ref{lem:hecke_coefs} provides asymptotics of images of $J$ under Hecke operators to be used in Section \ref{ssc:prf_1.1}. A much more natural description of the image of $J$ under the $m\tth$ Hecke operator in terms of the $m\tth$ Faber Polynomial is provided by the theory of replicability as described in Section \ref{ssc:replicability}.
\end{remark}
\noin In the case of Theorem \ref{thm:abelian_moonshine}, we specify that the McKay-Thompson series $T_{g}$ for any $g \in G$ be the Hauptmodul for the genus zero subgroup $\Gamma_{0}(\ord(g)).$  
\begin{lemma} \label{lem:haupt_coeefs}
The Hauptmodul for the genus zero subgroup $\Gamma_{0}(\ord(g))$ is given by
\begin{equation*}
    T_{g}(\tau) = q^{-1} + \sum_{n = 1}^{\infty} c_{g}(-1,n) q^{n}, 
\end{equation*}
where 
\begin{equation*}
    c_{g}(1,n) = 2 \pi \sqrt{\frac{1}{n}} \sum_{c > 0} \frac{K(-1,n,c \ord(g))}{c \ord(g)} \cdot I_{1}\left(\frac{4 \pi \sqrt{n}}{c  \ord(g)}\right)
\end{equation*}
\end{lemma}
We now sketch the proofs of Lemmas \ref{lem:hecke_coefs} and \ref{lem:haupt_coeefs}, which can be found in, for example, Chapter 6 of \cite{BFOR} or in \cite{hannah}.
\begin{proof}[Sketch of Proof of Lemmas \ref{lem:hecke_coefs} and  \ref{lem:haupt_coeefs}] Lemma \ref{lem:hecke_coefs} provides exact formulas for the coefficients of images of $J$ under Hecke operators. Lemma \ref{lem:haupt_coeefs} provides exact formulas for coefficeints of Hauptmoduln of genus zero congruence subgroups. The Poincar\'e series on $\Gamma_{0}(N)$ are natural generators for both images of $J$ under Hecke operators and Hauptmoduln of genus zero congruence subgroups as by construction they are invariant under the desired subgroups and have a pole of the desired order at infinity. Following notation as in Chapter 6 of \cite{BFOR}, weight zero Poincar\'e series on genus zero congruence subgroups with a pole of order $m$ at infinity are given by $$\P_{N}(\phi_{m}; \tau) := \sum_{\gamma \in \Gamma_{\infty} \backslash \Gamma_{0}(N)} \phi_{m}(\gamma \tau),$$ where writing $\tau = x + i y$ with $x,y \in \R, y > 0,$ we have $$\phi_{m}(\tau) := M_{0, \frac{1}{2}} ( 4 \pi m y) \exp{-2 \pi i m x },$$ where $M_{\mu, \nu}(w)$ is the classical $M$-Whittaker function. In general, the functions $\P_{N}(\phi_{m}; \tau)$ are harmonic Maass forms with both holomorphic and nonholomorphic parts, and so are not necessarily weakly holomorphic functions. However, any harmonic Maass form in the kernel of the $\xi$ differential operator $$\xi_{w} := 2 i y^{w} \cdot \frac{\overline{\partial}}{\partial \overline{z}}$$ is weakly holomorphic. It is known (see Chapter 5 in \cite{BFOR}) that $$\xi_{2-k} : H_{2 -k}(N) \rightarrow S_{k}(N),$$ where $H_{w}(N)$ denotes the space of weight $w$ harmonic Maass forms on $\Gamma_{0}(N)$ and $S_{w}(N)$ denotes the subspace of cusp forms. Since we are dealing wtih genus zero subgroups, the space of weight two cusp forms is empty and the weight zero Poincar\'e series on $\Gamma_{0}(N)$ we construct are assured of being weakly holomorphic. As shown in, for example, Chapter 6 of \cite{BFOR}, the Fourier expansion of the holomorphic part of $\P_{N}(\phi_{m}; \tau),$ normalized to remove constants, is given by $$ \P_{N}(\phi_{m}; \tau) = q^{-m} + \sum_{n=1}^{\infty} b_{m}^{+}(n) q^{n},$$ where $$b_{m}^{+}(n) = 2 \pi \sqrt{\frac{m}{n}} \cdot \sum_{\substack{c > 0 \\ N | c}} \frac{K(-m,n,c)}{c} \cdot I_{1}\left(\frac{4 \pi \sqrt{m n}}{c}\right).$$ We have that $\P_{N}(\phi_{m}; \tau)$ is therefore a modular function whose Fourier expansion is $q^{-m} + O(q)$ with no other poles. Such functions are unique, and Lemmas \ref{lem:hecke_coefs} and \ref{lem:haupt_coeefs} follow immediately.
\end{proof}
\noin Asymptotics for other functions we will use follow from Ingham's Tauberian theorem \cite{BM, Tessa, Ing}:
\begin{theorem}[Ingham]\label{thm:taub}
Let $f(q) = \sum_{n \geq 0} a(n) q^{n}$ be a  power series  with weakly increasing nonnegative coefficients and radius  of convergence equal to 1. If there are constants $ A > 0, \lambda, \alpha \in \R$ such that
    \begin{equation*}
    f(\exp{-\eps}) \sim \lambda \epsilon^{\alpha} \exp{a/\epsilon}
    \end{equation*}
as $\epsilon \to 0^{+}$, then as $n \to \infty$ we have that
    \begin{equation*}
    a(n) \sim \frac{\lambda}{2 \sqrt{\pi}} \frac{A^{\frac{\alpha}{2} + \frac{1}{4}}}{n^{\frac{\alpha}{2} + \frac{3}{4}}} \exp{2 \sqrt{A n}}.
    \end{equation*}
\end{theorem}
\noin This gives us the following asymptotic.
\begin{lemma}\label{lem:delta_asym}
Let $t \in \N$ and define $b_{m,t}(n)$ by 
\begin{equation*}
\pfrac{\Delta(m\tau)}{\Delta(\tau)}^t = \sum_{n=t(m-1)}^\infty b_{m,t}(n) q^n.
\end{equation*}
Then as $n \to \infty$ we have that
\begin{equation*}
b_{m,t}(n) \sim \frac{t\pfrac{m-1}{m}^{1/4}}{m^{12t}\sqrt{2}n^{3/4}} \cdot \exp{4\pi \sqrt{\frac{m-1}{m}tn}}.
\end{equation*}

\end{lemma}
\begin{proof}
Since $q=\exp{2\pi i \tau}$, we have $2\pi i \tau = -\eps$ in the notation of Theorem \ref{thm:taub}. Recall from \eqref{eq:eta_S} that $\eta(-\frac{1}{\tau}) = \sqrt{-i\tau} \cdot \eta(\tau)$, $\Delta(\tau) = \eta(\tau)^{24}$, and the asymptotic $\eta(i n)^k \sim \exp{-\frac{\pi k}{12} n}$. Together these imply that 
\begin{align*}
\pfrac{\Delta(m\tau)}{\Delta(\tau)}^t = \frac{\pfrac{2\pi}{m\eps}^{12t}}{\pfrac{2\pi}{\eps}^{12t}} \pfrac{\eta\pfrac{2\pi i}{m \eps}}{\eta\pfrac{2\pi i}{\eps}}^{24t} \sim \frac{1}{m^{12t}}\exp{\frac{4\pi^2 t(m-1)}{m} \cdot \frac{1}{\eps}}
\end{align*}
as $\eps \to 0^+$. Hence in the notation of Theorem \ref{thm:taub},  $\lambda = \frac{1}{m^{12t}}$, $\alpha = 0$, and $A = \frac{4\pi^2 t(m-1)}{m}$, 
which implies that
$$b_{m,t}(n) \sim \frac{1}{m^{12t}}\frac{t\pfrac{m-1}{m}^{1/4}}{\sqrt{2} n^{3/4}}\exp{4\pi \sqrt{\frac{m-1}{m}t} \sqrt{n}}.$$ 
\end{proof}

\subsection{Replicability}\label{ssc:replicability}

The formula \eqref{eq:hecke} for the Hecke operator applied to $f(\tau)$, where $\tau$ is seen as representing the lattice $\Lambda(\tau) = \langle 1, \tau\rangle \subset \C$, can be seen as the sum 
\begin{equation}\label{eq:hecke_sublattices}
f \mid nT(n) = \sum_{[\Lambda(\tau):\Lambda'] = n}f(\Lambda')
\end{equation}
by writing a basis for all sublattices of $\Lambda(\tau)$ of index $n$. If we define $$m(\Lambda') := \min\{d \in \N\subset \C \mid d\in \Lambda'\}$$ then replace $f$ by some $f^{(a)} = f^{(n/d)}$ inside the sum in \eqref{eq:hecke_sublattices} for different values $d = m(\Lambda')$ for different sublattices, we arrive a a natural generalization of the Hecke operators suited to a remarkable property of the Hauptmoduln called \emph{replicability}.

Explicitly, the function $J(\tau)$ and the other Hauptmoduln are unique in that they satisfy certain \emph{replication formulas} \cite{Conway-Norton}. This means that for a function $f$ there exist \emph{replicate powers} $f^{(a)}$ such that
\begin{equation}\label{eq:replicability_general}
F_m(f(\tau))= \sum_{\substack{a d = m 
\\ 0 \leq b < d}} f^{(a)} \left( \frac{a \tau + b}{d}\right).
\end{equation}
In the specific case of McKay-Thompson series for moonshine, this takes the form of the equation
\begin{equation}\label{eq:replicability}
F_m(T_g(\tau))= \sum_{\substack{a d = m 
\\ 0 \leq b < d}} T_{g^a} \left( \frac{a \tau + b}{d}\right)
\end{equation}
where $F_m(f)$ is a \emph{Faber polynomial}, the unique polynomial in $\Q[f]$, depending on the coefficients of $f$, such that if $f= q^{-1} + O(q)$ then $F_m(f) = q^{-m} + O(q)$. In particular, when $g = e$ then (\ref{eq:replicability}) says
\begin{equation}\label{eq:j_faber}
    F_m(J(\tau)) = J(\tau) \mid mT(m).
\end{equation}
In fact, a distinct generalization of this formula for $J$ exists for genus-zero Hauptmoduln (see Theorem 1.1 of \cite{lea_larson}). The first few Faber polynomials for $J$ are given by
\begin{align*}
F_1(J) &= J = q^{-1} + 196884q + O(q^2)
\\ F_2(J) &= J^2 - 393768 = q^{-2} + 42987520q + O(q^2)
\\ F_3(J) &= J^3 - 590652J - 64481280 = q^{-3} + 2592899910q + O(q^2).
\end{align*}
More generally, we have \cite{intro_replic}
$$
\frac{q\frac{\partial}{\partial q} J(\t)}{z - J(\t)} = \sum_{m=0}^\infty F_m(z) q^m = 1 + zq + (z^2- 393768)q^2 + \cdots $$
and in fact the denominator formula for the Monster Lie algebra is equivalent to (see Section 6 of \cite{traces})
$$J(\t) - J(z) = q^{-1} \exp{- \sum_{m=1}^\infty F_m(J(z)) \frac{q^m}{m} }. $$
Satisfying the replication formula is roughly equivalent to being a Hauptmodul. Precisely, Norton conjectured that replicable functions with integer coefficients are either of the form $q^{-1} + aq$ or are Hauptmoduln for genus zero congruence subgroups with width 1 at infinity. Cummings and Norton have proven that all such Hauptmoduln with rational coefficients are replicable \cite{cummingsnorton_rational_hauptmoduls_are_replicable}. This property of the Hauptmoduln is a cornerstone of much research into moonshine \cite{carnahan_gen_moonsihne_I}. Furthermore, replicability has applications in number theory: Zagier \cite{traces} makes striking use of the denominator formula in his proof of Borcherds' theorem on infinite product expansions of modular forms with Heegner divisor.

\section{Some Representation Theory for virtual moonshine}\label{sec:rep_thy}

\noin We prove that there exists ``virtual moonshine,'' in the sense below, for any finite group. 

\begin{lemma}\label{lem:pre_moon}
If $G$ is a finite group, then there exists $d \in \N$ and an infinite-dimensional, graded, possibly virtual $\C[G]$-module $$V^G = \bigoplus_{n \geq -d} V_n^G$$ such that for all $g$, we have that $$R_g := \sum_{n \geq -d} \Tr(g | V_n^G)q^n$$ is the $q$-series of a weakly-holomorphic form of weight $0$ and level strictly $\ord(g)$, with integral coefficients. Specifically, for a specified weakly holomorphic modular function $R_e$, we will show this result for $$R_g(q) = R_e + \ordG{G} f_{\ord(g)},$$ where $f_N \in M_0^!(N)$.
\end{lemma}

\subsection{Multiplicity Generating Functions}\label{ssc:multiplicity_gen_fns}

Fix a finite group $G$. Let $\chi_1,\ldots, \chi_n$ be all irreducible characters of $G$, with $\chi_1=\bbone$ the trivial character. The Schur orthogonality relations, as described in Chapter 18 of \cite{DF}, allow us to produce a multiplicity generating function for an infinite-dimensional graded $\C[G]$-module $V = \bigoplus_{n \gg -\infty}^\infty V_n^G$. If the $q$-graded trace for an element $g$ acting on $V$ is given by $R_g$ then by Schur orthogonality
\begin{align}\label{eq:mult_gen_fun}
F_i(q) &:= \sum_{n \gg - \infty} \mult_i(n) q^n =  \sum_{n \gg - \infty} \frac{1}{\ordG{G}}\sum_{g \in G} \overline{\chi_i(g)} \Tr(g |  V_n^G) q^n =\frac{1}{\ordG{G}}\sum_{g\in G}  \overline{\chi_i(g)}R_g(q)
\end{align}
is a Laurent series where $\mult_i(n)$ is the number of copies in $V_n^G$ of the irreducible representation with character $\chi_i$. Further, if each of the $R_g(q)$ is invariant under subgroups $\Gamma_g \leq \SL_2{(\Z)}$ then $F_i(q)$ will be invariant under a group containing their intersection $\bigcap_{g\in G} \Gamma_g$. 

The formula \eqref{eq:mult_gen_fun} for different $i$ can be nicely packaged in bulk quantities using the character table of $G$. Let $X$ be the matrix whose rows are indexed by irreducible characters $\chi_i$ and columns are indexed by conjugacy classes $[g_j]$, and the $(i,j)$ entry of $X$ is given by $\# [g_j]\overline{\chi_i(g_j)}$, where $\#[g]$ is the size of a conjugacy class. Then it follows that
\begin{equation}\label{eq:multiplicities_vector}
\begin{pmatrix}F_1(q)\\ \vdots \\ F_n(q)\end{pmatrix} = \frac{1}{\ordG{G}} X \begin{pmatrix}R_{g_1}\\ \vdots \\ R_{g_n}\end{pmatrix}.
\end{equation}

\noin To prove Lemma \ref{lem:pre_moon}, we must only show that the multiplicity generating functions given by \eqref{eq:mult_gen_fun} have integral coefficients.

\begin{proof}[Proof of Lemma \ref{lem:pre_moon}]
Define $R_e$ to be a Laurent series over $\Z$ where $R_e$ has leading term $q^{\d}$, and for $\ord(g)\neq 1$ define $$R_g := R_e + \ordG{G} f_{\ord(g)},$$ where $f_N \in M_0^!(N)$. If $\{\chi_1, \ldots, \chi_n\}$ are the irreducible characters of $G$, with $\chi_1$ the identity character, then $F_i(q)$ (the $i\tth$ multiplicity generating function) is given by 
\begin{align*}
F_i(q) &= \frac{1}{\ordG{G}} \sum_{g\in G} \overline{\chi_i(g)}R_g
&= \frac{1}{\ordG{G}} R_e \cdot \sum_{g\in G}\overline{\chi_i(g)} + \sum_{g\in G} \overline{\chi_i(g)}f_{\ord(g)} 
&= \begin{dcases} R_e + \overline{\sum_{g\in G} \chi_i(g)f_{\ord(g)}} & i = 1
\\ 0 + \overline{\sum_{g\in G} \chi_i(g)f_{\ord(g)}} & i \neq 1.
\end{dcases}
\end{align*}
Note that the $\chi_i(g)$ are in the ring of integers of $\Q(\xi_{\ordG{G}})$ where $\xi_{\ordG{G}}$ is a $\ordG{G}\tth$ root of unity. Furthermore,
$$\Gal(\Q(\xi_{\ordG{G}})/\Q) \simeq (\Z/\ordG{G}\Z)^\times$$ and $l\in (\Z/\ordG{G}\Z)^\times$ acts on the character table by the $l\tth$ power map in the columns. But $(l,\ordG{G}) = 1$ which means $(l,\ord(g)) = 1$ for all $g\in G$ so that $\ord(g^l) = \ord(g)$ and thus the coefficients of $F_i(q)$ are fixed by $l$, so they are rational algebraic integers and hence integers.
\end{proof}
\begin{remark}
This proof actually allows two conjugacy classes $[g_1]$ and $[g_2]$ with $\ord(g_1) = \ord(g_2)$ to have different associated $R_{g_i}$. What is actually required is that two conjugacy classes must have the same function $R_{g_i}$ if they lie in the same orbit of $\Gal(\Q(\xi_{\ordG{G}})/\Q)$ acting on the columns of the character table. 
\end{remark}

\section{Proof of Main Results}\label{sec:pf_thms}
\noin We prove Theorem \ref{thm:moonshine_always}, moonshine for all finite groups, in Section \ref{ssc:prf_1.1}; Theorem \ref{thm:order_bounds}, which bounds the power of primes which may divide the order of a group which has a moonshine module where the graded trace functions are Hauptmoduln for genus zero congruence subgroups, in Section \ref{ssc:order_bounds}; and Theorem \ref{thm:abelian_moonshine}, which classifies the finite abelian groups that have moonshine where the graded trace functions are Hauptmoduln for genus zero congruence subgroups, in Section \ref{ssc:prf_thm_1.3}.
\subsection{Proof of Theorem \ref{thm:moonshine_always}}\label{ssc:prf_1.1}
To prove Theorem \ref{thm:moonshine_always}, we prove the following theorem which provides modules with only finitely many virtual components. We will then make use of the Hecke algebra to get modules with no virtual components.
\begin{theorem}\label{thm:asymp_moonshine}
Let $G$ be a finite group. Then there exists a positive integer $\d$ and an infinite-dimensional graded virtual $\C[G]$-module $$V^G = \bigoplus_{n \in \{-\d\} \cup \N} V_n^G$$ such that for all $g \in G$, the McKay-Thompson series $\Tr(g | V^G)$ is a modular function strictly of level $\ord(g)$, and such that the homogeneous components $V_n^G$ are virtual modules for only finitely many $n$. 
\end{theorem}
\begin{remark}
Theorem \ref{thm:asymp_moonshine} may be thought of as a natural generalization of the work of Atkin-Fong-Smith \cite{AFS}, with more freedom in the choice of modular functions.
\end{remark}
\begin{proof}
For $n > 1$, let $t_n = \frac{n}{n-1}h$ where $h := \lcm\{\ord(g)-1 \mid g \in G, g \neq e\}$, making $t_g := t_{\ord(g)}$ always an integer (where we assume $g \neq e$).
Let us define $R_e = J \mid T(d)$ for $d \in \N$ to be chosen later, as well as
\begin{equation}
\bar{B}_{m,t_m} := m^{12t_m} \pfrac{\Delta(m\tau)}{\Delta(\tau)}^{t_m}
\end{equation}
for $m,t_m \in \N$, and
\begin{equation*}
R_g := R_e - \ordG{G} \cdot \bar{B}_{\ord(g),t_g},
\end{equation*}
for $g \neq e$.
Define $b_{\ord(g), t_{g}}(n)$ by
$$\bar{B}_{\ord(g),t_g} =  \sum_{n=t_g(\ord(g)-1)}^\infty \bar{b}_{\ord(g), t_{g}}(n) q^n,$$ 
and since $t_g$ and $\bar{b}_{\ord(g),t_g}$ will only depend on the order of $g$ we will sometimes write $\bar{b}_{v,t_v}$ where $v=\ord(g)$.

Lemma \ref{lem:delta_asym} implies, after multiplying through by $m^{12t_m}$, that
\begin{equation}\label{eq:cgasymptotic}
\bar{b}_{\ord(g), t_{g}}(n) \sim \frac{\left(\frac{\ord(g)-1}{\ord(g)} t_g\right)^{1/4}}{\sqrt{2} n^{3/4}} \cdot \exp{4\pi \sqrt{\frac{\ord(g)-1}{\ord(g)}t_gn}}.
\end{equation}

\noin The choice of the $t_g$ gives that this asymptotic is the same for all $g \neq e$. 

Now we will show the existence of $V$ using only the assumptions that the $\bar{b}_{\ord(g), t_{g}}(n)$ are all asymptotically the same as $n \to \infty$ independent of $g$, that they are positive, and that the depth $\d$ may be chosen arbitrarily high so $R_e$ dominates asymptotically. The Schur orthogonality relations yield that the multiplicities $\mult_i(n)$ of an irreducible representation with character $\chi_i$ in homogeneous component $V_n^G$ are given by \eqref{eq:mult_gen_fun}. By Lemma \ref{lem:pre_moon}, all $\mult_i(n)$ are integral, so it suffices to show nonnegativity.

\noin Let $\chi_i \neq \chi_1$ be nontrivial. Then $\sum_{g \in G} \overline{\chi_i(g)} = 0$, so $- \sum_{g \in G \setminus e} \overline{\chi_i(g)} = \dim(\chi_i)$. Letting
$$a_r = \sum_{\substack{g \in G \\ \ord(g)=r}} \overline{\chi_i(g)}$$
we then have 
\begin{equation}\label{eq:char}
-\sum_{\substack{r \big| \ordG{G} \\ r \neq 1}} a_r  = \dim(\chi_i) > 0
\end{equation}
and
\begin{equation}\label{eq:orders'}
\frac{1}{\ordG{G}}\sum_{g \in G} \overline{\chi_i(g)} R_g = - \sum_{\substack{r \big | \ordG{G} \\ r \neq 1}} a_r R_{r}.
\end{equation}
By \eqref{eq:cgasymptotic} and \eqref{eq:char}, \eqref{eq:orders'} implies that 
\begin{equation}\label{eq:mult_asymp}
\mult_i(n) = - \sum_{\substack{r \big| \ordG{G} \\ r \neq 1}} a_r b_{r,t_r}(n) \sim \dim(\chi_i) \frac{h^{1/4}}{\sqrt{2} n^{3/4}} \exp{4\pi \sqrt{h n}}.
\end{equation}
This shows that there are at most finitely many $n$ for which $\mult_i(n)$ is negative.

Now let $\chi_i = \chi_1$ be trivial. We must show that for all $n \geq 1$, the coefficients of
\begin{equation*}
\sum_{v \big| \ordG{G}} \ell(v) R_v
\end{equation*}
are nonnegative, where $\ell(v) := \# \{g \in G \mid \ord(g) = v\}$. Equivalently, we need 
\begin{equation*}
c_e(d,n) \geq \sum_{v \big| \ordG{G}, v > 1} \ell(v) \bar{b}_{v,t_v}(n)
\end{equation*} 
where $c_{e}(d,n)$ is as in Lemma \ref{lem:hecke_coefs}. The right hand side is asymptotic to 
\begin{equation*}
(\ordG{G}-1) \frac{h^{1/4}}{\sqrt{2} n^{3/4}} \exp{4\pi \sqrt{h n}}.
\end{equation*}
 Choosing $d > h$, we have by \eqref{eq:j_asym} that $c_{e}(d,n)$ dominates 
 
 \begin{equation*}
 \sum_{v \big| \ordG{G}, v > 1} \ell(v) \bar{b}_{v,t_v}(n)
 \end{equation*}
asymptotically, completing the proof.
\end{proof}
\noin This asymptotic result may be strengthened through a bounding argument.

\begin{proof}[Proof of Theorem \ref{thm:moonshine_always}]
Let $\hat{V}^G$ be the graded module guaranteed by Theorem \ref{thm:asymp_moonshine}, $\ell$ be its depth, and $R_g = \Tr(g | \hat{V}^G)$ be the corresponding McKay-Thompson series for $g \in G$. There are finitely many of these, which we will abuse notation and number as $R_1,\ldots,R_m$ where $m$ denotes the number of conjugacy classes of $G$ and these $R_i$ may be the same on conjugacy classes of the same order. We write $R_1=R_e$ and

\begin{equation*}
R_i = \sum_{n \gg -\infty} c_i(n)q^n
\end{equation*}
(note that $c_1(n) = c_e(\ell,n)$). For each character $\chi_i$, we have that the multiplicity $\mult_i(n)$ of the corresponding representation in the graded component $V_n^G$ is given by a linear function $L_i(c_1(n),\ldots,c_m(n))$. By Theorem \ref{thm:asymp_moonshine}, each $L_i(c_1(n),\ldots,c_m(n))$ is asymptotic to a positive function of $n$, therefore 
\begin{equation*}
    B = |\inf_{\substack{n \in \Z_{>0} \\ 1 \leq i \leq m}} L_i(c_1(n),\ldots,c_m(n))|
\end{equation*}
is finite. 
By the proof of Theorem \ref{thm:asymp_moonshine}, $L_i(c_1(n),\ldots,c_m(n))$ is asymptotic to a monotonically increasing unbounded function as $n \to \infty$. Therefore for each $\chi_i$ there exists some $N_i$ such that for all $n>N_i$, $L_i(c_1(n),\ldots,c_m(n)) > B$. Set $N' = \max_i N_i$.
\\\\
Now, for a prime $p$ coprime to $\ell$ we have 
\begin{align*}
R_i \mid pT(p) &= q^{-\ell p} + \sum_{n \geq 1} (pc_i(pn) + c_i(n/p))q^n,
\end{align*}
where $c_i(n/p) = 0$ when $p \nmid n$. In particular, any congruence mod $\ordG{G}$ satisfied by all coefficients of the $R_i$ is also satisfied by all coefficients of the $R_i \mid pT(p)$, so there is a possibly virtual module $V^G$ with graded traces $R_i \mid pT(p)$.
Fix $p>N'$; then for all $\chi_i$ and $n > 0$,
\begin{align*}
L_i(pc_1(pn)+c_1(n/p),\ldots,pc_m(pn)+c_m(n/p)) &= p L_i(c_1(pn),\ldots,c_m(pn)) + L_i(c_1(n/p),\ldots,c_m(n/p))\\
&\geq p L_i(c_1(pn),\ldots,c_m(pn)) - B \\
&\geq 0
\end{align*}
since $pn \geq p > N'$. For $n \leq 0$, the only nonzero coefficient of any $R_i \mid pT(p)$ is a $1$ in front of $q^{-p\ell}$. Therefore the coefficient in front of $q^{-p\ell}$ in each $L_i$ is $\ordG{G}$ for $\chi_i = \bbone$ and $0$ for $\chi_i$ nontrivial, so it is always nonnegative. This shows that the multiplicities of each irreducible representation in $V^G$ are nonnegative. Choosing $p$ also coprime to $\ordG{G}$, we have that $pT(p)$ does not affect the level of any $R_i$, so since the McKay-Thompson series of $\hat{V}^G$ are on the required level we have that those of $V^G$ are as well.
\end{proof}

\noin Note that applying $p T(p)$ to $R_{i}$ for any $p > N'$ results in a moonshine module of depth $p \ell$ with no negative multiplicities. Thus, we have infinitely many such modules. It is not clear from the above proof, however, if for a given group $G$ a moonshine module of \emph{every} depth greater than some specified depth $d$ necessarily exists, as $T_{e} = J \mid (\ell p) T(\ell p)$ for some $\ell$ as chosen in the proof of Theorem \ref{thm:asymp_moonshine} and some $p > N'$ as chosen in the proof above. Thus, we could imagine some very large prime $p' > h$ such that moonshine of this depth is not readily apparent from the above proof. A natural question is whether further work could provide insight on what depths of moonshine are possible for a given group $G$, especially establishing its (minimal) depth. Remark \ref{rmk:tensor_power_of_V_natural} shows that we can always bound the pole of $T_e(\tau)$ as $ i\tau \to \infty$ to have order $\le \ordG{G}$, although the graded dimension in this case is $J(\tau)^{\ordG{G}}$ and not $J \mid mT(m)$.

\subsection{Proof of Theorem \ref{thm:order_bounds}} \label{ssc:order_bounds}

Recall the notation from the statement of Theorem \ref{thm:abelian_moonshine} that $T_{\ord(g)}$ is the Hauptmodul for $\Gamma_{0}(\ord(g))$, and set
$$ T_{\ord(g)} := \sum_{n = -1}^{\infty} c_{g}(1,n) q^n.$$
We use this notation for consistency with Lemma \ref{lem:hecke_coefs}, which gives asymptotics for $c_{e}(1,n)$ (and $c_{e}(m,n)$ generally). All groups considered in this Section must, by the hypotheses of Theorems \ref{thm:order_bounds} and \ref{thm:abelian_moonshine}, have only elements such that $\Gamma_{0}(\ord(g))$ is genus-zero, so we may speak of its Hauptmodul in all cases.

\begin{proof}[Proof of Theorem \ref{thm:order_bounds}]
Let $G$ be a finite group satisfying the hypotheses. For $g \in G$, $\Gamma_0(\ord(g))$ must have genus $0$ and hence we must have $\ord(g) \in \{1,2,\ldots,10,12,13,16,18,25\}$. The primes dividing elements of this set are $2,3,5,7$ and $13$; if another prime $p$ divides $\ordG{G}$, then by the first Sylow theorem, $G$ contains a $p$-group, which contradicts the hypothesis that $\Gamma_0(\ord(g))$ always has genus $0$. Hence $\ordG{G}$ is divisible only by $2,3,5,7$ and $13$.
\\\\
\noin By the Schur orthogonality argument of Section \ref{ssc:multiplicity_gen_fns}, integrality of the multiplicities of an irreducible representation in each homogeneous component is equivalent to certain congruences between the functions $T_1,T_2,\ldots,T_{25} \pmod{\ordG{G}}$. In particular, the integrality of the multiplicities of the trivial character $\chi_1$ is equivalent to a congruence modulo $\ordG{G}$ where $T_1$ has coefficient $1$, which is equivalent to a set of congruences of the form
\begin{equation}\label{eq:1congruence}
T_1 + a_2 T_2 + \ldots + a_{25} T_{25} \equiv 0 \pmod{p^{\ord_p(\ordG{G})}},
\end{equation}
for $p \in \{2,3,5,7,13\}$, where $a_i = \ell(i) =  \# \{g \in G \mid \ord(g) = i\}$. Therefore, to obtain a bound on the highest power of $p$ which may divide $\ordG{G}$, it suffices to show that there do not exist any congruences of the form\footnote{It is necessary to consider congruences of this form because there will always be congruences among $T_1,T_2,\ldots,T_{25}$ modulo $p^N$ given by multiplying a congruence modulo a lower power of $p$ by the appropriate power of $p$; however, mandating that the first coefficient is $1$ excludes such cases.} of \eqref{eq:1congruence} modulo $p^N$ for some $N$, for then no such congruences can exist modulo $p^n$ for $n \geq N$.

The final part of the proof consists in computationally finding all congruences between $T_1,T_2,\ldots,T_{25}$ modulo powers of $p \in \{2,3,5,7,13\}$. This was carried out with Sage\footnote{Code available at \url{https://github.com/nvafa/moonshine-congruences}.} \cite{sage} by computing the left kernel of the matrix
\begin{equation}
\begin{pmatrix}
c_1(1,-1) & c_1(1,1) & \ldots & c_1(1,N) \\
c_2(1,-1) & c_2(1,1) & \ldots & c_2(1,N) \\
\vdots & \vdots & \ddots & \vdots \\
c_{25}(1,-1) & c_{25}(1,1) & \ldots & c_{25}(1,N) \\
\end{pmatrix}
\end{equation}
We note that, since we are looking for the lowest power of $p$ for which there is \emph{not} a congruence of the form in \eqref{eq:1congruence}, one may choose any $N$ (at the risk of possibly getting a worse bound); however, $N$ was chosen based on Lemma \ref{lem:sturm} (Sturm bound) so that the congruences obtained below are guaranteed to hold for all coefficients of $T_1,T_2,\ldots,T_{25}$. The congruences obtained are shown below, where the power of $p$ is the largest such that a congruence of the form in \eqref{eq:1congruence} exists:
\begin{align*}
0 & \equiv T_1 + 7940351 T_2 + 22091520 T_4 + 4308992 T_8 + 32768000 T_{16} & \pmod{2^{25}} \\
0 & \equiv T_1 + 527795 T_3 + 3645 T_9 & \pmod{3^{12}} \\
0 & \equiv T_1 + 13124 T_5 + 2500 T_{25} & \pmod{5^{6}} \\
0 & \equiv T_1 - T_7 & \pmod{7^4} \\
0 & \equiv T_1 - T_{13} & \pmod{13^2}.
\end{align*}
This immediately yields the desired bounds.
\end{proof}

\begin{remark}
While the above proof is by casework, the fact that the uniform bound $\ord_p(\ordG{G}) \leq \frac{24}{p-1}$ holds for all odd primes is suggestive. An explicit construction is needed to show whether the bound is sharp, and it is hoped that further restrictions, coming from either integrality or nonnegativity of the multiplicities, force the sharper bound $\ord_2(\ordG{G}) \leq 24 = \frac{24}{2-1}$ so that the formula $\ord_p(\ordG{G}) \leq \frac{24}{p-1}$ holds for all primes.
\end{remark}

\noin For abelian groups, it is possible to extend the techniques of the previous proof to obtain explicit formulas for the multiplicity generating functions in terms of the group, allowing a complete classification of abelian groups with (possibly virtual) moonshine. Further analysis shows that the virtual modules obtained are actually modules in all but a few cases.

\subsection{Proof of Theorem \ref{thm:abelian_moonshine}} \label{ssc:prf_thm_1.3}

We begin with several lemmas regarding moonshine for abelian groups. First, in Lemma \ref{lem:raw_congs}, we provide a concrete description of the congruences that must be satisfied amongst Hauptmoduln in the case of finite abelian groups. Recall that the \emph{conductor} of a character $\chi$ on $\Z/m\Z$ is the smallest $m'$ such that $\chi$ factors through $\Z/m'\Z$. Given a character $\chi$ on an abelian group $G$ and a prime $p \big| \ordG{G}$, define the \emph{$p$-max conductor} of $\chi$ to be the largest power $p^\ell$ of $p$ such that $\chi$ restricted to some $\Z/p^n\Z$ summand of $G$ has conductor $p^\ell$. The usefulness of this notion lies in the fact that the multiplicity generating series of a character on an abelian group $G$ is entirely determined by its $p$-max conductors for each prime dividing $\ordG{G}$. The lemma below explicitly computes, for any abelian group character, the $q$-series which will be equal to its multiplicity generating function if there is a moonshine module with graded traces equal to the appropriate Hauptmoduln. Therefore the problem of existence of such a module is reduced to checking whether the Fourier coefficients of this function are nonnegative integers.

\begin{lemma}\label{lem:raw_congs}
Let $A = \{p_1,\ldots,p_k\}$ be a finite set of primes, $G = \prod_{p \in A} \prod_{j=1}^{h_p} \left(\Z/p^j\Z\right)^{r_{p,j}}$ a finite abelian group, $\chi$ an irreducible character on $G$ with $p$-max conductor $n_p$ for $p \in A$. Furthermore, define 
\begin{equation*}
\num_p(n,t) := \begin{cases}
0 & \text{ if } t>n \\
-\pi_p(t) & \text{ if } t = n \\
\pi_p(t+1)-\pi_p(t) & \text{ if } t < n,
\end{cases}
\end{equation*}
where
\begin{equation*}
\pi_p(t) := p^{r_{p,1}+2r_{p,2}+\ldots+(t-1)r_{p,t-1}+\ldots+(t-1)r_{p,h_p}}
\end{equation*}
for $t>1$ and by convention $\pi_p(1)=1$ and $\pi_p(0)=0$. Then 
\begin{equation}\label{eq:explicit_form}
\frac{1}{\ordG{G}}\sum_{g\in G}  \overline{\chi(g)}T_{\ord(g)} = \frac{1}{\ordG{G}} \sum_{\mathbf{t} \in \Z_{\geq 0}^{k}} \left(\prod_{p \in A} \num_p(n_p,t_p) \right) T_{\mathbf{t}}
\end{equation}
where $T_{\mathbf{t}} := T_{\prod_{p \in A} p^{t_p}}$.
\end{lemma}
\begin{proof}
By the assumption that $\Tr(g|V^G) = T_{\ord(g)}$ and the Schur orthogonality argument of Section \ref{sec:rep_thy},  
\begin{equation*}
F_r(q) = \frac{1}{|G|}\sum_{d\big| \ordG{G}} \left(\sum_{\substack{g \in G \\ \ord(g)=d}} \chi(g)\right) T_g.
\end{equation*}
First consider the case where $G$ is a $p$-group and $d=p^t$. Because a character on a direct product of groups is just a product of characters on the individual groups, 
\begin{equation*}\label{eq:orders}
\sum_{\substack{g \in G \\ \ord(g)=p^t}} \chi(g) = \prod_{j=1}^{h_p} \prod_{i=1}^{r_{p,j}} \left(\sum_{a=0}^{p^{\text{min}(t,j)}} \chi_{i,j}(a \cdot p^{j - \text{min}(t,j)})\right) - \prod_{j=1}^{h_p} \prod_{i=1}^{r_{p,j}} \left(\sum_{a=0}^{p^{\text{min}(t-1,j)}} \chi_{i,j}(a \cdot p^{j - \text{min}(t-1,j)})\right)
\end{equation*}
where $\chi_{i,j}$ are the characters on the components $\Z/p^j\Z$ of $G$; the first product counts elements with each coordinate having order at most $p^t$, from which we subtract a product corresponding to all elements with order at most $p^{t-1}$. Each term 
\begin{equation*}
\sum_{a=0}^{p^{\text{min}(t,j)}} \chi_{i,j}(a \cdot p^{j - \text{min}(t,j)})
\end{equation*}
is equal to $p^{\text{min}(t,j)}$ if $\chi_{i,j}$ takes value $1$ on each of $0,p^{j-\text{min}(t,j)},\ldots,(p^{\text{min}(t,j)}-1)p^{j-\text{min}(t,j)}$, and $0$ otherwise. Therefore, letting $p^{n}$ be the largest conductor of the $\chi_{i,j}$, we have
\begin{equation*}
\prod_{j=1}^{h_p} \prod_{i=1}^{r_{p,j}} \left(\sum_{a=0}^{p^{\text{min}(t,j)}} \chi_{i,j}(a \cdot p^{j - \text{min}(t,j)})\right) - \prod_{j=1}^{h_p} \prod_{i=1}^{r_{p,j}} \left(\sum_{a=0}^{p^{\text{min}(t-1,j)}} \chi_{i,j}(a \cdot p^{j - \text{min}(t-1,j)})\right) = \num_p(n,t).
\end{equation*}
Thus
\begin{equation*}
F_r(q) = \sum_{t \in \Z_{\geq 0}} \num_p(n,t)T_{p^{t}}
\end{equation*}
for some $0 \leq n \leq h_p$, and this is exactly the form specified. If $G$ is a product of abelian $p$-groups for $p \in A$, then a character of $G$ is a product of characters on each $p$-group, and therefore by the same calculation the coefficient of $T_{\mathbf{t}}$ in the expansion of $F_r(q)$ is the product $\prod_{p \in A} \num_p(n_p,t_p)$ where $p^{n_p}$ is the $p$-max conductor of $G$ for each $p \in A$. 
This shows the result for all finite abelian groups.
\end{proof}

\noin Showing the existence of a graded moonshine module for $G$ is equivalent to showing that the coefficients of 
$$\sum_{\mathbf{t} \in \Z_{\geq 0}^k} \left(\prod_{p \in A} \num_p(n_p,t_p) \right) T_{\mathbf{t}}$$
are positive and $0$ modulo $\ordG{G}$, so that they are integral when dividing out by $\ordG{G}$.
The nonnegativity of the coefficients of the multiplicity generating functions $F_r$ must be shown directly,  and the following lemma shows that the congruences modulo $\ordG{G}$ are equivalent to a simpler set of congruences with all dependence on the number of $\Z/p^n\Z$ summands concentrated in the modulus. In the next lemma we refer to the set of all congruences 
$$\sum_{\mathbf{t} \in \Z_{\geq 0}^k} \left(\prod_{p \in A} \num_p(n_p,t_p) \right) T_{\mathbf{t}} \equiv 0 \pmod{\ordG{G}}$$
obtained from Lemma \ref{lem:raw_congs} as ($C1$) collectively. 

\begin{lemma} \label{lem:nice_congs}
Fix notation as in Lemma \ref{lem:raw_congs}, and for $S \subset A$ denote $T_{\prod_{p \in S} p^{n_p}}$ by $T_{S,\mathbf{n}}$. Then the set of congruences ($C1$) is satisfied if and only if, for every $P \subset A$ and $\mathbf{n}=(n_{p_1},\ldots,n_{p_k}) \in \Z^k$ with $1 \leq n_p \leq h_p$, the congruence 
\begin{equation*}
\sum_{S \subset P} (-1)^{|S|} T_{S,\mathbf{n}} \equiv 0 \bigpmod{\prod_{p \in P} \prod_{j = n_p}^{h_p} p^{r_{p,j}(j-n_p+1)}}
\end{equation*}
is satisfied. We refer to this set of congruences collectively as ($C2$). 
\end{lemma}
\begin{proof}
First assume the congruences ($C1$) hold. For each $\mathbf{n}=(n_{p_1},\ldots,n_{p_k}) \in \Z^k$ with $1 \leq n_p \leq h_p$, there exists an irreducible character on $G$ with $n_{p_i}$ as its $p_i$-max conductor for each $i$. Hence ($C1$) contains all congruences of the form 
\begin{equation*}
\sum_{\mathbf{t} \in \Z_{\geq 0}^k} \left(\prod_{p \in S} c_p(n_p,t_p) \right) T_{\mathbf{t}} \equiv 0 \pmod{\ordG{G}}
\end{equation*}
for each tuple $\mathbf{n} = (n_{p_1},\ldots,n_{p_k}) \in \Z_{\geq 1}^k$.

Fix notation
\begin{equation*}
    R_{\mathbf{n}} = \sum_{\mathbf{t} \in \Z_{\geq 0}^k} \left(\prod_{p \in A} \num_p(n_p,t_p) \right) T_{\mathbf{t}}
\end{equation*}
and for a subset $P \subset A$
\begin{equation*}
    N_{P,\textbf{e}} = \sum_{S \subset P} (-1)^{|S|} T_{S,\textbf{e}},
\end{equation*}
where $\textbf{e} = (e_{p_1},\ldots,e_{p_k}) \in \N^k$ and we take $e_{p_i} = h_{p_i}+1$ for all $p_i \not \in P$ and $1 \leq e_{p_i} \leq h_{p_i}$ for all $p_i \in P$. Additionally for $e,n \in \N$, let 
\begin{equation*}
    \ell_p(e,n) = \begin{cases}
    0 & \text{ if } n>e \\
    1 & \text{ if } n=e \\
    \pi_p(e)\left(\frac{1}{\pi_p(n)} - \frac{1}{\pi_p(n+1)}\right) & \text{ if }n<e
    \end{cases}.
\end{equation*}
Then we claim that 
\begin{equation*}
\sum_{\mathbf{n} \in \N^k} \left(\prod_{p \in A} \ell_p(e_p,n_p)\right) R_{\mathbf{n}} = \left(\prod_{p \in A} \pi_p(e_p)\right) N_{P,\textbf{e}}.
\end{equation*}
First, substituting the expression for $R_{\mathbf{n}}$ and collecting coefficients, the left hand side is equal to
\begin{align*}
\sum_{\mathbf{t} \in \Z_{\geq 0}^k} T_{\mathbf{t}} \sum_{\mathbf{t} \leq \mathbf{n} \leq \textbf{e}} \prod_{p \in A} \num_p(n_p,t_p) \ell_p(e_p,n_p) 
 =\sum_{\mathbf{t} \in \Z_{\geq 0}^k} T_{\mathbf{t}} \prod_{p \in A} \left(\sum_{t_p \leq n_p \leq e_p}  \num_p(n_p,t_p) \ell_p(e_p,n_p)\right),
\end{align*}
where the $\leq$ in the index of summation denotes the obvious partial order on $\Z^k$. Now, if $t_p=e_p$ then 
\begin{equation*}
\sum_{t_p \leq n_p \leq e_p}  \num_p(n_p,t_p) \ell_p(e_p,n_p) = \num_p(e_p,e_p)\ell_p(e_p,e_p) = \pi_p(e_p)
\end{equation*}
and if $t_p = 0$ then 
\begin{equation*}
\sum_{t_p \leq n_p \leq e_p}  \num_p(n_p,t_p) \ell_p(e_p,n_p) = \sum_{1 \leq n_p \leq e_p} \ell_p(e_p,n_p) = \pi_p(e_p)
\end{equation*}
since $\pi_p(1)=1$ and the series telescopes. However, if $0 \neq t_p \neq e_p$ then we have 
\begin{align*}
\sum_{t_p \leq n_p \leq e_p}  c_p(n_p,t_p) \ell_p(e_p,n_p) &= c_p(t_p,t_p)\ell_p(e_p,t_p) + \sum_{t_p + 1 \leq n_p \leq e_p}(\pi_p(t_p+1)-\pi_p(t_p))\ell_p(e_p,n_p) \\
& = -\pi_p(t_p) \cdot \pi_p(e_p)\left( \frac{1}{\pi_p(t)} - \frac{1}{\pi_p(t+1)}\right) + (\pi_p(t_p+1) - \pi_p(t_p)) \cdot \frac{\pi_p(e_p)}{\pi_p(t_p+1)} \\
& = 0.
\end{align*}
Therefore the coefficient of $T_{\mathbf{t}}$ is 
\begin{equation*}
\begin{cases}
-\pi_p(e_p) & \text{ if } t_p=e_p \\
\pi_p(e_p) & \text{ if } t_p = 0 \\
0 & \text{ otherwise}
\end{cases}
\end{equation*}
from which it follows that 
\begin{equation*}
\sum_{\mathbf{n} \in \N^k} \left(\prod_{p \in A} \ell_p(e_p,n_p)\right) R_{\mathbf{n}} = \left(\prod_{p \in A} \pi_p(e_p)\right) N_{S,\textbf{e}}
\end{equation*}
as desired. Therefore if the congruences ($C1$) are satisfied then we have 
\begin{equation*}
\left(\prod_{p \in A} \pi_p(e_p)\right) N_{S,\textbf{e}} \equiv 0 \pmod{\ordG{G}}
\end{equation*}
and dividing out by $\prod_{p \in A} \pi_p(e_p)$ we obtain the congruences ($C2$), showing one direction.
\\\\
Now suppose the congruences $(C2)$ are satisfied. Because the matrix of $\ell_p(\textbf{e},\mathbf{n})$ is upper-triangular (for a suitable ordering), has integer entries, and has $1$'s on the diagonal, it follows that the adjugate of the matrix is its inverse, and hence that it is invertible over $\Z$. Therefore the $R_{\mathbf{n}}$'s are expressible as integer linear combinations of the $\left(\prod_{p \in A} \pi_p(e_p)\right) N_{S,\textbf{e}}$'s. Hence if the latter are congruent to $0$ mod $\ordG{G}$ then the former are as well.
\end{proof}

\noin Armed with these explicit descriptions of the congruences and nonnegativity conditions which must be checked for an abelian group to have a moonshine module, we may computationally classify all such groups and thus prove Theorem \ref{thm:abelian_moonshine}.

\begin{proof}[Proof of Theorem \ref{thm:abelian_moonshine}]
To show integrality, by Lemma \ref{lem:nice_congs}, it suffices to show that 
\begin{equation}\label{eq:nice_congs}
\sum_{S \subset P} (-1)^{|S|} T_{S,\mathbf{n}} \equiv 0 \bigpmod{\prod_{p \in P} \prod_{j = n_p}^{h_p} p^{r_{p,j}(j-n_p+1)}}
\end{equation}
holds for all subsets $P \subset A$. Since the graded trace functions $T_N$ are required to be Hauptmoduln for $\Gamma_0(N)$ for which $X_0(N)$ is genus zero, we can consider only the groups with element orders contained in \eqref{eq:genus_zero_N}. Directly applying Lemma \ref{lem:nice_congs}, we can find the highest powers $n$ of each $p \in P$ for which 
\begin{equation*}
\sum_{S \subset P} (-1)^{|S|} T_{S,\mathbf{n}} \equiv 0 \pmod{p^{n}}
\end{equation*}
and use this to get bounds on the exponents $r_{p,j}$ by \eqref{eq:nice_congs}.
\\\\
We cannot verify that all coefficients of the $T_{S,\mathbf{n}}$ satisfy such a congruence, but we can apply Lemma \ref{lem:sturm} to reduce to a finite check which may be done by computer. We need to check only the first $[\text{SL}_2(\Z) : \Gamma_0(N)] = N\prod_{p \mid N} (1 + p^{-1})$, where $N$ is the least common multiple of all the levels in a given congruence. This is maximized when we consider congruences with $T_9$ and $T_2$, in which $[\text{SL}_2(\Z) : \Gamma_0(18)] = 36$. Thus, checking the first 36 coefficients of all of our congruences uniformly is sufficient for our purposes. 
Using Sage \cite{sage}, \textsc{Appendix A} gives the list of all the congruences we obtained, with the maximum possible moduli. From these congruences and Lemma \ref{lem:nice_congs}, we can instantly read off the highest rank finite abelian groups we can get, which show that the groups listed in Table \ref{table:abelian_groups} have infinite dimensional graded modules with integral multiplicities of irreducible representations.
\\\\Having verified the integrality of the $\mult_i(n)$, we show nonnegativity  to verify the existence of honest modules. By the triangle inequality we have that 
$$
    \mult_i(n) \geq \frac{1}{\ordG{G}}\left(|\overline{\chi_{i}(e)} c_{e}(1,n)|-\sum_{\substack{[g] \in G \\ [g] \neq [e]}} \#[g] |\overline{\chi_{i}(g)} c_{g}(1,n)|\right).
$$
So, to show nonnegativity of $\mult_i(n)$ it suffices to show
\begin{equation} \label{eq:tri_ineq}
    c_{e}(1,n)-\sum_{\substack{[g] \in G \\ [g] \neq [e]}} \#[g] |c_{g}(1,n)| > 0
\end{equation}
for all $n$ sufficiently large, and then check smaller $n$ by computer. 
From Lemma \ref{lem:haupt_coeefs} follow the bounds:
\begin{align} 
    c_{e}(1,n) &\geq \frac{2 \pi}{\sqrt{n}} \left( K(-1,n,1) \cdot I_{1}(4 \pi \sqrt{n}) - \sum_{c > 1} \left|\frac{K(-1,n,c)}{c}\right| \cdot \left|I_{1}\left(\frac{4 \pi \sqrt{n}}{c}\right)\right|\right)  \label{eq:lo_bnd_ce}
    \\|c_{g}(1,n)| &\leq \frac{2 \pi}{\sqrt{n}} \sum_{c > 0} \left|\frac{K(-1,n,c \ord(g))}{c \ord(g)}\right| \cdot \left|I_{1}\left(\frac{4 \pi \sqrt{n}}{c \ord(g)}\right) \right|. \label{eq:up_bnd_cg}
\end{align}
Moreover, we have the following elementary bounds:
\begin{align*}
    \left|I_{1}(x) \right| &< \frac{x}{2} e^{\frac{x^2}{4}} & (0 < x)
    \\\left|I_{1}(x)\right| &< \frac{2x}{3} & (0 < x < 1)
    \\\left|\frac{K(-1,n,c)}{c}\right| &< \frac{2}{c^\frac{1}{4}}.
\end{align*}
Applying these elementary bounds to \eqref{eq:lo_bnd_ce} gives
$$
    c_e(1,n) \geq \frac{2 \pi}{\sqrt{n}} \left(I_{1}(4 \pi \sqrt{n}) - \sum_{c  = 2}^{\ceil{4 \pi \sqrt{n}}} \frac{2}{c^\frac{1}{4}} I_{1}\left(\frac{4 \pi \sqrt{n}}{c}\right) - \sum_{c = \ceil{4 \pi \sqrt{n}} + 1}^{\infty}\frac{16 \pi \sqrt{n}}{3 c^{5/4}} \right),
$$
which simplifies to
\begin{equation}\label{eq:j_lower}
    c_{e}(1,n) \geq \frac{2 \pi}{\sqrt{n}} I_{1}(4 \pi \sqrt{n}) - 8\cdot2^{\frac{3}{4}} \pi^{2} I_{1}(2 \pi \sqrt{n}) - \frac{64 \sqrt{2} \pi^{\frac{7}{4}}}{3 n^{\frac{1}{8}}}.  
\end{equation}
Applying these elementary bounds to \eqref{eq:up_bnd_cg} gives
$$
    |c_{g}(1,n)| \leq \frac{2 \pi}{\sqrt{n}}\left(\sum_{j=1}^{\ceil[\big]{\frac{4 \pi \sqrt{n}}{\ord(g)}}} \frac{2}{(j \ord(g))^{1/4}} I_{1}\left(\frac{4 \pi \sqrt{n}}{j \ord(g)}\right) + \sum_{ \ceil[\big]{ \frac{4 \pi \sqrt{n}}{\ord(g)} }+1}^{\infty} \frac{16 \pi \sqrt{n}}{3 j^{\frac{5}{4}} \ord(g)^{\frac{5}{4}}}\right),
$$
which simplifies to 
\begin{equation}\label{eq:g_upper}
    |c_{g}(1,n)| \leq \frac{4 \pi}{\sqrt{n} \ord(g)^{\frac{1}{4}}} \left(1 + {\frac{4 \pi \sqrt{n}}{\ord(g)}}\right) I_{1}\left(\frac{4 \pi \sqrt{n}}{\ord(g)}\right) + \frac{64 \sqrt{2} \pi^{\frac{7}{4}}}{3 \ord(g) n^{\frac{1}{8}}}.
\end{equation}
Therefore, plugging into \eqref{eq:tri_ineq} gives
\begin{equation} \label{eq:pos_rel}
\begin{split}
    & \frac{2 \pi}{\sqrt{n}} I_{1}(4 \pi \sqrt{n}) - 8 \cdot 2^{\frac{3}{4}} \pi^{2} I_{1}(2 \pi \sqrt{n}) - \frac{64 \sqrt{2} \pi^{\frac{7}{4}}}{3 n^{\frac{1}{8}}}
    \\&- \sum_{\substack{[g] \in  \Conj(G) \\ [g] \neq [e]}} \#[g] \left(\frac{4 \pi}{\sqrt{n} \ord(g)^{\frac{1}{4}}} \left(1 + {\frac{4 \pi \sqrt{n}}{\ord(g)}}\right) I_{1}\left(\frac{4 \pi \sqrt{n}}{\ord(g)}\right) + \frac{64 \sqrt{2} \pi^{\frac{7}{4}}}{3 \ord(g) n^{\frac{1}{8}}}\right) > 0,
\end{split}
\end{equation}
which implies $\mult_i(n)$ positive. Using the result \cite{LN}
$$
    \frac{I_{1}(x)}{I_{1}(y)} < e^{x-y} \frac{y}{x} \hspace{1cm} (x < y)
$$
simplifies \eqref{eq:pos_rel} to
\begin{equation}\label{eq:bound_no_al}
\begin{split}
    & I_{1}(2 \pi \sqrt{n})\left(\frac{\pi e^{2 \pi \sqrt{n}}}{\sqrt{n}} - 8 \cdot 2^{\frac{3}{4}} \pi^2- \sum_{\substack{[g] \in  \Conj(G) \\ [g] \neq [e]}}  \#[g] \frac{4 \pi \ord(g)^{\frac{3}{4}}}{\sqrt{n}} \left(1 + \frac{4 \pi \sqrt{n}}{\ord(g)}\right) e^{4  \pi \sqrt{n} (\frac{1}{\ord(g)} - \frac{1}{2})}\right) 
    \\&- \frac{64 \sqrt{2} \pi^{\frac{7}{4}}}{3 n^{\frac{1}{8}}}\left(1 + \sum_{\substack{[g] \in  \Conj(G) \\ [g] \neq [e]}}  \frac{\#[g]}{\ord(g)}\right) > 0.
\end{split}
\end{equation}
Now, by the power series for $I_1$ we have that $I_{1}(2 \pi \sqrt{n})$ is monotonically increasing in $n$. Examining the coefficient 
$$
    \frac{\pi e^{2 \pi \sqrt{n}}}{\sqrt{n}} - 8 \cdot 2^{\frac{3}{4}} \pi^2- \sum_{\substack{[g] \in \Conj(G) \\ [g] \neq [e]}}  \#[g] \frac{4 \pi \ord(g)^{\frac{3}{4}}}{\sqrt{n}} \left(1 + \frac{4 \pi \sqrt{n}}{\ord(g)}\right) e^{4  \pi \sqrt{n} \left(\frac{1}{\ord(g)} - \frac{1}{2}\right)}
$$
we have that $\frac{\pi e^{2 \pi \sqrt{n}}}{\sqrt{n}}$ is positive and is monotonically increasing for $n \geq 1$, while 
$$
\sum_{\substack{[g] \in  \Conj(G) \\ [g] \neq [e]}}  \#[g] \frac{4 \pi \ord(g)^{\frac{3}{4}}}{\sqrt{n}} \left(1 + \frac{4 \pi \sqrt{n}}{\ord(g)}\right) e^{4  \pi \sqrt{n} \left(\frac{1}{\ord(g)} - \frac{1}{2}\right)}
$$
decreases monotonically, using the fact that $\ord(g) \geq 2$ and so the exponent
$$4  \pi \sqrt{n} \left(\frac{1}{\ord(g)} - \frac{1}{2}\right)$$
is increasingly negative with $n$. Therefore the first of the two summands of \eqref{eq:bound_no_al} is monotonically increasing for $n \geq 1$, so since the second summand negative and monotonically decreasing, we have that if the inequality of \eqref{eq:bound_no_al} is satisfied for some $N \geq 1$ then it is satisfied for all $n \geq N$. Therefore it suffices to compute this $N$ for a given group and then check nonnegativity up to that bound, which we did with Sage \cite{sage}. A computation shows that we can choose $N = 100$ uniformly for all the groups we are considering.
\end{proof}
\noin We note that our requirement that all multiplicities be positive introduced only one additional constraint not already provided by the requirement that all multiplicities be integral; namely, the requirement that $a + b + c + d \leq 13$ for groups of the form $(\Z/2\Z)^a \times (\Z/4\Z)^b \times (\Z/8\Z)^c \times (\Z/16\Z)^d$ (see Table \ref{table:abelian_groups}).

\section{Examples and Discussion}\label{sec:example}

\noin In Sections \ref{ssc:ex_z7} and \ref{ssc:ex_s4}, we go through some examples of groups $G$ that explicitly show the existence of the desired $\C[G]$-modules. In Section \ref{ssc:discussion}, we discuss our results and present further questions.

\subsection{Elementary Abelian 7-group of rank 4}\label{ssc:ex_z7}

\noin Consider $G = (\Z/7\Z)^4$, which is not a subgroup of $\M$ (\cite{wilson_communication}; see also Theorem 7 of \cite{monster_subgroup}). Theorem \ref{thm:abelian_moonshine} says that there is an infinite dimensional, graded $\C[G]$-module where the graded trace functions for $g \in G$ is $T_{\ord(g)}$, the Hauptmodul for $X_0(\ord(g))$.
\\\\
Since $G$ has only elements of order $1$ or $7$, we need only consider the Hauptmoduln
\begin{align*}
T_1(\t) &= J(\t) = q^{-1} + 196884q + O(q^2),
\\T_7(\t) &= \frac{\eta(\t)^4}{\eta(7 \t)^4} + 4 = q^{-1} + 2q + O(q^2).
\end{align*}
Let $\mult_i(n)$ denote the multiplicity of the $i \tth$ irreducible representation, where $1 \leq i \leq 2401$ and $i=1$ denotes the trivial representation. From \eqref{eq:mult_gen_fun}, we know
$$F_i(q) = \sum_{n \geq -1} \mult_i(n)q^n = \frac{1}{7^4} \sum_{g \in G} \overline{\chi_i(g)} T_{\ord(g)}(q).$$
Since $(\Z/7\Z)^4$ is abelian, its character table is easy to compute. It will be the fourth iterated Kronecker product of the character table for $\Z/7\Z$ with itself, and for brevity, the character table for $\Z/7\Z$ is
	\begin{figure}[H]
	\caption{Character Table for $\Z/7\Z$}
	\vspace{2ex}
	\begin{tabular}{c|ccccccc}
		$\Z/7\Z$ &$e$&$1$&$2$&$3$&$4$&$5$&$6$\\
		\hline
		\(\chi_1\) &$ 1 $&$ 1 $&$ 1 $&$ 1 $&$ 1 $&$ 1 $&$ 1$\\
		\(\chi_2\) &$ 1 $&$ \zeta $&$ \zeta^2 $&$ \zeta^3 $&$	\zeta^4 $&$ \zeta^5 $&$ \zeta^6$\\
		\(\chi_3\) &$ 1 $&$ \zeta^2 $&$ \zeta^4 $&$ \zeta^6 $&$ \zeta $&$ \zeta^3 $&$ \zeta^5$\\
		\(\chi_4\) &$ 1 $&$ \zeta^3 $&$ \zeta^6 $&$ \zeta^2 $&$ \zeta^5 $&$ \zeta $&$ \zeta^4$\\
		\(\chi_5\) &$ 1 $&$ \zeta^4 $&$ \zeta $&$ \zeta^5 $&$ \zeta^2 $&$ \zeta^6 $&$ \zeta^3$\\
		\(\chi_6\) &$ 1 $&$ \zeta^5 $&$ \zeta^3 $&$ \zeta $&$ \zeta^6 $&$ \zeta^4 $&$ \zeta^2$\\
		\(\chi_7\) &$ 1 $&$ \zeta^6 $&$ \zeta^5 $&$ \zeta^4 $&$ \zeta^3 $&$ \zeta^2 $&$ \zeta$\\
		
	\end{tabular}
	\end{figure}
\noin where $\zeta := \exp{2 \pi i / 7}$. By the iterated Kronecker product of the above table, or directly from Schur's orthogonality relations, we get
$$ \sum_{\substack{g \in G \\ \ord(g) = 7}} \overline{\chi_i(g)} = \begin{cases}
7^4-1 & i = 1\\
-1 & i \neq 1.
\end{cases}$$
Thus
$$
    F_i(q) = \begin{dcases}
     \frac{T_1(q) +(7^4-1) T_7(q)}{7^4} & i = 1\\
     \frac{T_1(q) - T_7(q)}{7^4} & i \neq 1.\\
    \end{dcases}
$$
Since $T_1$ and $T_7$ have integral coefficients, in order for these multiplicities to be integral for all cases, it suffices to show that 
$$ T_1 \equiv T_7 \pmod{7^4}.$$
Lemma \ref{lem:sturm} says that we only need to the check the first
$$[\text{SL}_2(\Z) : \Gamma_0(7)] - 1 = 7 \prod_{\text{prime }p \mid 7} (1 + p^{-1}) - 1= 7$$
coefficients. We can do this explicitly. Sage \cite{sage} gives us that
\begin{equation*}\label{eq:7_rels}
\begin{split}
\frac{T_1 - T_7}{7^4} &= 82q + 8952q^2 + 359975q^3 + 8432260q^4 + 138776610q^5 
\\& + 1770938484q^6 + 18599331142q^7 + 167218195192q^8 + O(q^9).
\end{split}
\end{equation*}
Now, to show nonnegativity, it suffices to show that the coefficients of $T_1(q) - T_7(q)$ and $T_1(q) + (7^4-1) T_7(q)$ are nonnegative. By the triangle inequality,
$$\mult_i(n) \geq |c_1(1,n)| - (7^4-1) |c_7(1,n)|.$$
We recall that \eqref{eq:j_lower} gives a lower bound on $|c_1(1,n)|$, and equation (\ref{eq:g_upper}) gives an upper bound on $|c_7(1,n)|$:
\begin{align*}
    |c_{1}(1,n)| &= c_{1}(1,n) \geq \frac{2 \pi}{\sqrt{n}} I_{1}(4 \pi \sqrt{n}) - 8\cdot2^{\frac{3}{4}} \pi^{2} I_{1}(2 \pi \sqrt{n}) - \frac{64 \sqrt{2} \pi^{\frac{7}{4}}}{3 n^{\frac{1}{8}}},  
    \\|c_7(1,n)| &\leq \frac{4 \pi}{\sqrt{n} 7^{\frac{1}{4}}} \left(1 + {\frac{4 \pi \sqrt{n}}{7}}\right) I_{1}\left(\frac{4 \pi \sqrt{n}}{7}\right) + \frac{64 \sqrt{2} \pi^{\frac{7}{4}}}{21 n^{\frac{1}{8}}}.
\end{align*}
Putting these together and using bounds on the first Bessel function, we get

\begin{equation*}
\begin{split}
    \mult_i(n) &\geq I_{1}(2 \pi \sqrt{n})\left(\frac{\pi e^{2 \pi \sqrt{n}}}{\sqrt{n}} - 8\cdot 2^{\frac{3}{4}} \pi^2 - (7^4-1) \frac{4 \pi 7^{\frac{3}{4}}}{\sqrt{n}} \left(1 + \frac{4 \pi \sqrt{n}}{7}\right) e^{4  \pi \sqrt{n} (\frac{1}{7} - \frac{1}{2})}\right)
    \\&- \frac{64 \sqrt{2} \pi^{\frac{7}{4}}}{3 n^{\frac{1}{8}}}\left(1 + \frac{7^4-1}{7}\right),
\end{split}
\end{equation*}
where the right hand side is monotonically increasing. For $n = 2$, the right hand side is positive, and a quick calculation shows that $\mult_i(n) \geq 0$ for the first two coefficients, and therefore for all $1 \leq i \leq 7^4$ and $n \geq -1$.
\\\\
To see that this representation converges to the regular representation as $n \rightarrow \infty$, we can look at the proportion of the multiplicity of trivial and nontrivial representations, which should be uniform since all irreducible representations of an abelian group have dimension 1. Define
$$\delta(\mult_i(n)) := \frac{\mult_i(n)}{\sum_{j = 1}^{2401}\mult_j(n)}$$
and note that $\mult_i(n) = \mult_j(n)$ when $i,j > 1$. We illustrate these asymptotics explicitly:
\begin{figure}[H]
\caption{Proportions of Irreducible Representations in $(\Z/7\Z)^4$ moonshine module}
\vspace{2ex}
\begin{tabular}{|c|c|c|c|c|}
\hline
$n$ & $\delta(\mult_1(n))$ & $\delta(\mult_{2}(n))$ & $\hdots$ & $\delta(\mult_{2401}(n))$\\
\hline
$-1$ & $1$ & $0$ & $\hdots$ & $0$\\

$0$ & -- & -- & $\hdots$ & --\\

$1$ & $4.2664... \times 10^{-4}$ & $4.1648... \times 10^{-4}$ & $\hdots$ & $4.1648... \times 10^{-4}$\\

$2$ & $4.1686... \times 10^{-4}$ & $4.1649... \times 10^{-4}$ & $\hdots$ & $4.1649... \times 10^{-4}$\\

$3$ & $4.1649... \times 10^{-4}$ & $4.1649... \times 10^{-4}$ & $\hdots$ & $4.1649... \times 10^{-4}$\\

\vdots & \vdots & \vdots &  & \vdots\\

$\infty$ & $1/2401$ & $1/2401$ & $\hdots$ & $1/2401$ \\
\hline
\end{tabular}
\end{figure}
\noin Also, since we have chosen the graded trace functions to be $J$ and the Hauptmodul for $X_0(7)$, these functions are replicable (see Section \ref{ssc:replicability} for more details).

\subsection{Symmetric group on 4 letters}\label{ssc:ex_s4}

To demonstrate our proof for Theorem $\ref{thm:moonshine_always}$, we will construct a moonshine module for the nonabelian group $G = S_4$.
\\\\
Before proceeding, we note that $1 \leq \ord(g) \leq 4$ for all $g \in G$. One can construct a moonshine module that assigns the graded trace functions to be Hauptmoduln for $X_0(\ord(g))$. That is, \emph{$S_4$ has depth one}. Using an argument similar to the previous example, one can check the desired congruences up to the Sturm bound, which can be uniformly chosen to be $23$ since all modular functions can be viewed on level $12$. A bounding argument like the one above shows that the necessary coefficients of the required relations are nonnegative.
\\\\
Also, with regards to the construction in Remark \ref{rmk:tensor_power_of_V_natural}, a quick computation shows that the graded trace assignments
\begin{align*}
    [e] &\mapsto J(\t)^4
    \\ [(12)] &\mapsto J(2\t) J(\t)^2
    \\ [(12)(34)] &\mapsto J(2\t)^2
    \\ [(123)] &\mapsto J(3\t) J(\t)
    \\ [(1234)] &\mapsto J(4\t)
\end{align*}
also give a valid $\C[S_4]$-module.
\\\\
We now demonstrate our method of proof for Theorem $\ref{thm:moonshine_always}$. In order for the asymptotics for our (non-identity) graded trace functions to be equal, we choose the smallest $h \in \N$ such that
$$t_{\ord(g)} = \frac{h \ord(g)}{\ord(g)-1} \in \N$$
for all $g \in G$, $\ord(g) > 1$. Thus we choose $h = \lcm \{ \ord(g) - 1: g \neq e\} = 6$. Since we need $R_1(\t) = J(\t) \mid dT(d)$ to dominate the functions
$$\bar{B}_{m,{t_m}}(\t) = m^{12 t_m}  \left(\frac{\Delta(m\t)}{\Delta(\t)}\right)^{t_m}$$
asymptotically, we initially choose $d = 7 > 6$, that is
$$R_1(\t) = J(\t) \mid dT(d) = q^{-7} + 44656994071935q + O(q^2).$$
We now define the graded trace functions
$$R_m(\t) = R_1(\t) - 24 \bar{B}_{m,t_m}(\t).$$
Lemma \ref{lem:pre_moon} says that these graded trace functions will satisfy the congruences required for the multiplicities of representations to be integral, no matter our choice of $R_1(\t)$ and $\bar{B}_{m,t_m}(\t)$, as long as they are Laurent series with coefficients in $\Z$ (which these are). We can look at the character table for $S_4$ to verify this:

\begin{figure}[H]
\caption{Character Table for $S_4$}
\vspace{2ex}
\begin{tabular}{c|ccccc}
        & $(1)$ & $(6)$ & $(3)$ & $(8)$ & $(6)$\\
		$S_4$ &$e$&$(12)$&$(12)(34)$&$(123)$&$(1234)$\\
		\hline
		\(\chi_1\) &$ 1 $&$ 1 $&$ 1 $&$ 1 $&$ 1$\\
		\(\chi_2\) &$ 1 $&$ -1 $&$ 1 $&$ 1 $&$ -1 $ \\
		\(\chi_3\) &$ 2 $&$ 0 $&$ 2 $&$ -1 $&$ 0 $ \\
		\(\chi_4\) &$ 3 $&$ 1 $&$ -1 $&$ 0 $&$ -1$ \\
		\(\chi_5\) &$ 3 $&$ -1 $&$ -1 $&$ 0 $&$ 1$ \\
		
	\end{tabular}
	\end{figure}
\noin In fact, all characters are integers here so the coefficients of
\begin{align*}
F_i(q) &= \frac{1}{24} \sum_{g \in G} \overline{\chi_i(g)} R_g(\t)
\end{align*}
will be integral by construction of our $R_m(\t)$. For nonnegativity, we have to check that five $q$-series will be positive, each corresponding to different irreducible characters of $S_4$. Denoting $\bar{B}_{m,t_m}$ by $\bar{B}_m$ for brevity, the irreducible representation multiplicity generating functions will be
\begin{align*}
F_i(q) &= \frac{1}{24} \begin{cases}
R_1 + 9R_2 + 8R_3+ 6R_4 & i = 1\\
R_1 -3R_2 + 8R_3 - 6R_4  & i = 2\\
2R_1 + 6R_2 - 8R_3  & i = 3\\
3R_1 + 3R_2 - 6R_4  & i = 4\\
3R_1 - 9R_2 + 6R_4  & i = 5\\
\end{cases}
\\&= \begin{cases}
R_1 - 9\bar{B}_2 - 8\bar{B}_3 - 6\bar{B}_4 & i = 1\\
3\bar{B}_2 - 8\bar{B}_3 + 6\bar{B}_4  & i = 2\\
-6\bar{B}_2 + 8\bar{B}_3  & i = 3\\
- 3\bar{B}_2 + 6\bar{B}_4  & i = 4\\
 9\bar{B}_2 - 6\bar{B}_4  & i = 5.\\
\end{cases}
\end{align*}
As we can see, these are concretely integral. In terms of nonnegativity, for $i \neq 1$,
$$\mult_i(n) \sim \dim(\chi_i) \frac{6^{1/4}}{\sqrt{2} n^{3/4}} \exp{4\pi \sqrt{6 n}}$$
(see \eqref{eq:mult_asymp} for details), which is asymptotically positive. By \eqref{eq:j_asym} for $i = 1$,
\begin{align*}
\mult_1(n) &\sim \frac{d^{1/4}}{\sqrt{2}n^{3/4}} \exp{4\pi \sqrt{dn}} -  23\frac{6^{1/4}}{\sqrt{2} n^{3/4}} \exp{4\pi \sqrt{6 n}}
\\&\sim \frac{d^{1/4}}{\sqrt{2}n^{3/4}} \exp{4\pi \sqrt{dn}}
\end{align*}
since we have chosen $d = 7 > 6$. Thus, we can choose some $N \in \N$ for which $\mult_i(n) \geq 0$ for all $n > N$ and all $i \in \{1,2,3,4,5\}$. In order to make $\mult_i(n) \geq 0$ for all $n \in \N$, we can hit all graded trace functions with the $113 \tth $ normalized Hecke operator. We can illustrate the proportions of the first few irreducible representations explicitly:

\begin{figure}[H]
\caption{Proportions of Irreducible Representations in asymptotically trivial $S_4$ module}
\vspace{2ex}
\begin{tabular}{|c|c|c|c|c|c|}
\hline
$n$ & $\delta(\mult_1(n))$ & $\delta(\mult_{2}(n))$ & $\delta(\mult_{3}(n))$ & $\delta(\mult_{4}(n))$ & $\delta(\mult_{5}(n))$\\
\hline
$-791$ & $1$ & $0$ & $0$ & $0$ & $0$\\

$1$ & $9.99... \times 10^{-1}$ & $3.38... \times 10^{-15}$ & $6.42... \times 10^{-14}$ & $6.75... \times 10^{-14}$ & $1.42... \times 10^{-13}$\\

$2$ & $9.99... \times 10^{-1}$ & $6.30... \times 10^{-19}$ & $1.34... \times 10^{-18}$ & $1.97... \times 10^{-18}$ & $2.06... \times 10^{-18}$\\

\vdots & \vdots & \vdots & \vdots & \vdots & \vdots\\

$100$ & $9.99... \times 10^{-1}$ & $1.17... \times 10^{-116}$ & $2.35... \times 10^{-116}$ & $3.53... \times 10^{-116}$ & $3.53... \times 10^{-116}$\\

\vdots & \vdots & \vdots & \vdots & \vdots & \vdots\\

$\infty$ & $1$ & $0$ & $0$ & $0$ & $0$\\
\hline
\end{tabular}
\end{figure}
\noin Addressing Remark \ref{rmk:repr_dists}, there is a simple fix we can do to make this module asymptotically regular. Instead of taking
$$R_1 = (J \mid 7T(7)) \mid 113T(113) = J \mid 791T(791),$$
we can define
$$R_e = 24 \cdot (J \mid pT(p)) + J \mid 791T(791),$$
for a prime $p > 791$. In particular, we can choose $p=797$. Now, the left hand term of $R_e$ will asymptotically dominate all other terms, making the coefficient of $R_e$, which is $\dim(\chi_i)$, the proportion of $i\tth$ irreducible representations as $n \to \infty$. We can now illustrate the proportions of irreducible representations explicitly:

\begin{figure}[H]
\caption{Proportions of Irreducible Representations in asymptotically regular $S_4$ module}
\vspace{2ex}
\begin{tabular}{|c|c|c|c|c|c|}
\hline
$n$ & $\delta(\mult_1(n))$ & $\delta(\mult_{2}(n))$ & $\delta(\mult_{3}(n))$ & $\delta(\mult_{4}(n))$ & $\delta(\mult_{5}(n))$\\
\hline
$-797$ & $1/10$ & $1/10$ & $1/5$ & $3/10$ & $3/10$\\
$-791$ & $1$ & $0$ & $0$ & $0$ & $0$\\

$1$ & $1.22... \times 10^{-1}$ & $9.74... \times 10^{-2}$ & $1.94... \times 10^{-1}$ & $2.92... \times 10^{-1}$ & $2.92... \times 10^{-1}$\\

$2$ & $1.13... \times 10^{-1}$ & $9.85... \times 10^{-2}$ & $1.97... \times 10^{-1}$ & $2.95... \times 10^{-1}$ & $2.95... \times 10^{-1}$\\

\vdots & \vdots & \vdots & \vdots & \vdots & \vdots\\

$100$ & $1.00... \times 10^{-1}$ & $9.99... \times 10^{-2}$ & $1.99... \times 10^{-1}$ & $2.99... \times 10^{-1}$ & $2.99... \times 10^{-1}$\\

\vdots & \vdots & \vdots & \vdots & \vdots & \vdots\\

$\infty$ & $1/10$ & $1/10$ & $1/5$ & $3/10$ & $3/10$\\
\hline
\end{tabular}
\end{figure}
\subsection{Discussion}\label{ssc:discussion}

The notion of depth defined in this paper \eqref{def:depth} is essentially bounding the order of pole at $i\infty$ for $\Tr(e | V^G)$. It is natural to ask more generally about the full polar divisors of the graded trace functions. In particular, for a given $m$, we considered graded trace functions of the form
$$ R'_{m,d} := J(\t) \mid dT(d) - \ordG{G} \cdot m^{12t_m} \left(\frac{\Delta(m \t)}{\Delta(\t)}\right)^{t_m}.$$
It is easy to see that $J(\t) \mid dT(d)$ will have a pole of order $d$ at each cusp. For the right-hand-side term, Theorem 1.65 of \cite{webofmodularity} gives
$$ \ord\left(\frac{\Delta(m \t)}{\Delta(\t)} ; \frac{a}{b}\right) = \frac{b^2-m}{b \gcd(b, \frac{m}{b})}$$
where $a,b \in \N$, $(a,b) = 1$, and $b | m$. Therefore, the order of the pole of $R'_{m,d}$ at a cusp $a/b$ is given by
$$-\ord\left(R'_{m,d} ; \frac{a}{b}\right) = \max\left(t_m\frac{m-b^2}{b \gcd(b, \frac{m}{b})},d\right)$$
and in particular, these modular functions will have a pole of order greater than $d$ at $a/b$ if and only if $b^2 < m$. If we write
$$(R'_{m,d}) = (R'_{m,d})_0 - (R'_{m,d})_\infty,$$
where the divisors $(R'_{m,d})_0, (R'_{m,d})_\infty$ are nonnegative, then we have
$$ (R'_{m,d})_\infty = \sum_{\substack{b | m \\ (a,b) = 1 \\ a \Mod{(b,\frac{m}{b})} }} \max\left(t_m\frac{m-b^2}{b \gcd(b, \frac{m}{b})},d\right) \cdot \frac{a}{b},$$
where $a,b \in \N$ and $\dfrac{a}{b}$ denotes a cusp of $X_0(m)$. Note that this description of cusps comes from Proposition 2.6 of \cite{iwaniec}. One can derive a similar formula for the polar parts of divisors to the functions we use to prove Theorem \ref{thm:moonshine_always}, where we apply a Hecke operator to all graded trace functions.
\\\\
In this light, one way in which monstrous moonshine is so special is that \emph{all} of its graded trace functions have a simple pole at $i \infty$ and are holomorphic at \emph{all other cusps}. The existence of these modular functions in itself is rare since there are finitely many subgroups of $\SL_2(\R)$ commensurable with $\SL_2(\Z)$ that give rise to genus zero modular curves; the fact that these Hauptmoduln are naturally related to conjugacy classes of $\M$ is, in this view, astonishing.
\\\\
\noin We pose a few natural questions based on our results:

\begin{que}
        Theorem \ref{thm:order_bounds} gives restrictions on the finite groups with moonshine modules for which graded traces are Hauptmoduln of congruence subgroups. Theorem \ref{thm:abelian_moonshine} actually classifies the finite abelian groups with this property. What more explicit descriptions can be given if we consider all finite groups?
\end{que}

\begin{que}
    Frenkel-Lepowsky-Meurman \cite{FLM1, FLM2} construct a vertex operator algebra for the monster module. Theorem \ref{thm:abelian_moonshine} gives that $(\Z / 5\Z)^5$ and $(\Z / 7\Z)^4$, which are not subgroups of the Monster, have moonshine with Hauptmoduln as graded trace functions. The argument in Remark \ref{rmk:tensor_power_of_V_natural} gives one way of constructing a vertex operator algebra with an action of these groups, but is there a vertex operator algebra structure on the modules constructed in Theorem $\ref{thm:abelian_moonshine}$? Perhaps, since $(\Z/5\Z)^5$ is a subgroup of the Weyl group $W(A_4^6)$ and $(\Z / 7\Z)^4$ is a subgroup of $W(A_6^4)$ there will be a vertex operator algebra with an action of these Weyl groups. What about other $W(N)$ for $N$ a Niemeier lattice? In fact, the 71 vertex operator algebras with graded dimension $J(\t) + k$ for some constant $k$ are known (see \cite{hol_voa, hol_voa_24} and references cited therein). It is natural to consider whether these groups are represented in these 71 examples, where these graded traces appear except with different constant terms.
\end{que}

\begin{que}\label{qu:atkin_lehner}
    Theorems \ref{thm:order_bounds} and \ref{thm:abelian_moonshine} describe groups which have Hauptmoduln of congruence subgroups as graded traces. What if we allow Atkin-Lehner involutions so that we can consider more modular curves with genus zero, as in the case of monstrous moonshine? We note that the set of 171 functions of Hauptmoduln which show up as graded trace functions on the monster module has rank 163 \cite{Conway-Norton}, which leads to difficulties in generalizing the proof of Theorem \ref{thm:order_bounds}. We expect that a refinement of these methods and the requirement of nonnegativity of multiplicities can be used to bound the order of possible groups, and hope that bounds achieved in this way reveal $\M$ to be maximal in an appropriate sense.

\end{que}

\begin{que}
  Theorem \ref{thm:moonshine_always} tells us that the depth of a group is always finite. What are the groups with depth $1$ moonshine? When can the graded trace functions be chosen to be replicable, as in monstrous moonshine?
\end{que}

\begin{que}
This paper tells us that we can construct $\C[G]$-modules with asymptotically trivial or asymptotically regular representations, which are in some sense the two extremal distributions we would expect. Which other distributions of irreducible representations can be realized as $\C[G]$-modules?
\end{que}

\begin{que}
The existence of moonshine modules hinges on the existence of congruences between modular forms. Modulo $p$, essentially the only such congruence is $E_{p-1} \equiv 1\pmod{p}$. But $E_{p-1}$ is essentially congruent to the supersingular polynomial for $p$, which as Ogg observed (and famously offered a bottle of Jack Daniels to anyone who could explain) \cite{ogg}, splits over $\F_p$ if and only if $p \big| \ordG{\M}$. Could this congruence be used to provide an answer to the Jack Daniels problem which does not employ the genus zero property of the normalizer of $\Gamma_0(p)$ in $\SL_2(\Z)$ for these primes?
\end{que}

\newpage

\section*{Appendix A. Hauptmoduln Congruences}
\noin Here we list the maximal moduli of congruences obtained to prove Theorem \ref{thm:abelian_moonshine}. Note that here, $T_N$ denotes the normalized Hauptmodul for $X_0(N)$.

\begin{align*}
0 &\equiv T_1 - T_2 & \pmod{2^{16}}
\\&\equiv T_1 - T_4 & \pmod{2^8}
\\&\equiv T_1 - T_8 & \pmod{2^4}
\\&\equiv T_1 - T_{16} & \pmod{2^2}
\\&\equiv T_1 - T_3 & \pmod{3^9}
\\&\equiv T_1 - T_9 & \pmod{3^3}
\\&\equiv T_1 - T_5 & \pmod{5^5}
\\&\equiv T_1 - T_{25} & \pmod{5^1}
\\&\equiv T_1 - T_7 & \pmod{7^4}
\\&\equiv T_1 - T_{13} & \pmod{13^2}
\\&\equiv T_1 - T_2 - T_3 + T_6 & \pmod{2^4 3^3}
\\&\equiv T_1 - T_4 - T_3 + T_{12} & \pmod{2^2 3^2}
\\&\equiv T_1 - T_2 - T_9 + T_{18} & \pmod{2^2 3^1}
\\&\equiv T_1 - T_2 - T_5 + T_{10} & \pmod{2^3 5^2}
\end{align*}


\newpage

\begin{thebibliography}{99}

\bibitem{AL} A. O. L. Atkin and J. Lehner, \emph{Hecke operators on $\Gamma_0(m)$},
Math. Ann. \textbf{185} (1970), 134-160.

\bibitem{lea_larson} L. Beneish and H. Larson, \emph{Traces of Singular Values of Hauptmoduln}, arXiv:1407.4479 [math.NT]

\bibitem{Bor1} R. Borcherds, Vertex algebras, Kac-Moody algebras, and the Monster, Proc. Nat. Acad. Sci. U.S.A. 83
(1986), no. 10, 3068-3071.

\bibitem{Bor2} R. Borcherds, \emph{Monstrous moonshine and monstrous Lie superalgebras}, Invent. Math. \textbf{109, No. 2} (1992), 405-444.

\bibitem{BFOR} K. Bringmann, A. Folsom, K. Ono, and L. Rolen, \emph{Harmonic Maass Forms and Mock Modular Forms: Theory and Applications}. Colloquium publications, Amer. Math. Soc. (2017), to appear.

\bibitem{BM} K. Bringmann and K. Mahlburg, \emph{Asymptotic formulas for stacks and unimodal sequences}, Journal of Combinatorial Theory \textbf{A} (2014), no. 126, 194-215.

\bibitem{carnahan_gen_moonsihne_I} S. Carnahan, ``Generalized moonshine I: Genus zero functions,'' \emph{Algebra and Number Theory} \textbf{4:6} (Dec., 2008) 649-679.

\bibitem{carnahan_gen_moonshine_IV} S. Carnahan, \emph{Generalized Moonshine IV: Monstrous Lie algebras}, (2012).

\bibitem{CDH1}
M. Cheng, J. Duncan, and J. Harvey, \emph{Umbral Moonshine}, Commun. Number Theory Phys. 8 (2014), no. 2.

\bibitem{CDH2}
M. Cheng, J. Duncan, and J. Harvey, \emph{Umbral Moonshine and the Niemeier Lattices}, Research in the Mathematical Sciences 1 (2014), no. 3.

\bibitem{Conway-Norton} J. H. Conway and S. P. Norton, \emph{Monstrous Moonshine}, Bull. London Math Soc. \textbf{11} 1979, no. 308-339. MR 554399 (81j:20028)

\bibitem{cummingsnorton_rational_hauptmoduls_are_replicable} C. J. Cummins and S. P. Norton, \emph{Rational Hauptmodul are replicable},
Canad. J. Math. \textbf{47} (1995) 1201-1218.

\bibitem{Tessa} T. Cotron, R. Dicks, and S. Fleming, \emph{Asymptotics and congruences for partition functions which
arise from finitary permutation groups}, arXiv:1606.09074 [math.NT] (2016).

\bibitem{DLM} C. Dong, H. Li, G. Mason,
\emph{Modular-invariance of trace functions in orbifold theory and generalized Moonshine},
Comm. Math. Phys., \textbf{214} (2000), pp. 1-56.

\bibitem{DF} D. S. Dummit and R. M. Foote, \emph{Abstract algebra}, Wiley, 2004.

\bibitem{moonshine_survey} J. F. R. Duncan, M. Griffin, and K. Ono, \emph{Moonshine}, Research in the Mathematical Sciences, \textbf{2} (2015), A11.

\bibitem{umbral_moonshine} J. F. R. Duncan, Michael J. Griffin, Ken Ono, \emph{Proof of the umbral moonshine conjecture}, Research in the Mathematical Sciences, \textbf{2} (2015), A26.

\bibitem{duncan_rademacher_sums_gravity} J. F. R. Duncan, I. Frenkel, \emph{Rademacher sums, moonshine and gravity}, Commun. Number Theory Phys., \textbf{5}, (2011).

\bibitem{onan} J. F. R. Duncan, M. H. Mertens and K. Ono, \emph{O'Nan Moonshine and Arithmetic}, arXiv:1702.03516.

\bibitem{EOT}
T. Eguchi, H. Ooguri, and Y. Tachikawa, \emph{Notes on the K3 Surface and the Mathieu group M24},
Exper. Math. \textbf{20} (2011), 91-96.

\bibitem{hol_voa}
J. van Ekeren, S. M\"oller and N. R. Scheithauer. Construction and classification of holomorphic vertex operator algebras. (arXiv:1507.08142v2 [math.RT]), 2015.

\bibitem{FLM1} I. Frenkel, J. Lepowsky, and A. Meurman, \emph{A natural representation of the Fischer-Griess Monster with the modular function J as character}, Proc. Nat. Acad. Sci. U.S.A., \textbf{81} (1984), no. 10, Phys. Sci., 3256-3260.

\bibitem{FLM2} I. Frenkel, J. Lepowsky, and A. Meurman, \emph{A moonshine module for the Monster}, Vertex operators in mathematics and physics (Berkeley, Calif., 1983), Math. Sci. Res. Inst. Publ., vol. 3, Springer, New York,  1985, 231-273.

\bibitem{FLM3}
I. Frenkel, J. Lepowsky, and A. Meurman, \emph{Vertex Operator Algebras and the Monster}, Pure and Applied Mathematics, vol. 134. Academic Press Inc., Boston, MA (1988). MR 90h:17026

\bibitem{LN} A. Laforgia, P. Natalini, \emph{Some inequalities for modified Bessel functions}, J. Inequal. Appl. (2010), Art. ID 253035, 10 pp

\bibitem{muchadoaboutmatthieu} T. Gannon, \emph{Much ado about Matthieu}, (2012).

\bibitem{Ing} A. E. Ingham, \emph{A Tauberian theorem for partitions}, Annals of Mathematics \textbf{42} (1941), no. 5,
1075-1090.

\bibitem{iwaniec} H. Iwaniec, \emph{Topics in Classical Automorphic Forms}, Grad. Studies in Math., (\textbf{17}) AMS (1997). 

\bibitem{hol_voa_24}
C. H. Lam and X. Lin. A holomorphic vertex operator algebra of central
charge 24 with weight one Lie algebra $F_{4,6}A_{2,2}$. (arXiv:1612.08123v1 [math.QA]), 2016.

\bibitem{hannah} H. Larson. \emph{Coefficients of Mckay-Thompson series and distributions of the moonshine module.} Proceedings of the American Mathematical Society \textbf{144.10} (2016), 4183-4197. 

\bibitem{mason}
G. Mason, \emph{Finite groups and modular functions}, The Arcata Conference on Representations of Finite Groups (Arcata, Calif., 1986), Proc. Sympos. Pure Math., vol. 47, Amer. Math. Soc., Providence, RI, 1987, With an appendix by S. P. Norton, pp. 181-210. MR 933359 (89c:11066)

\bibitem{intro_replic} J. McKay and A. Sebbar, \emph{Replicable functions: an introduction}, Frontiers in number theory, physics and geometry II, 373-386, Springer, Berlin, 2007.

\bibitem{norton_generalized}
S. P. Norton, \emph{Generalized Moonshine}, Proc. Symp. Pure Math 47 (1987), 208–209.

\bibitem{ogg} A. Ogg, \emph{Automorphisms de courbes modulaires}, Sem. Delange-Pisot-Poitou, Th\'eorie des nombres, \textbf{16}, no. 1 (1974-1975), exp. no. 7, 1-8.

\bibitem{webofmodularity}  K. Ono, \emph{The Web of Modularity: Arithmetic of the Coefficients of Modular Forms and
q-series}, CBMS Regional Conference Series in Mathematics, \textbf{102}, Amer. Math. Soc., Providence, RI, 2004

\bibitem{queen}
L. Queen, \emph{Modular functions arising from some finite groups}, Math. Comp. 37 (1981), no. 156,
547-580. MR 628715 (83d:20008)

\bibitem{AFS} S. D. Smith, \emph{On the head characters of the Monster simple group}, Finite Groups---coming of age (Montreal, Que, 1982), Contemp. Math., vol. 45, Amer. Math. Soc., Providence, RI, 1985, pp. 303-313. MR 822245 (87h:20037)

\bibitem{sage}
W. A. Stein et al. \emph{Sage Mathematics Software (Version 7.6)}. The Sage Development Team, 2017. \url{http://www.sagemath.org}.

\bibitem{sturm}
J. Sturm, \emph{On the congruence of modular forms}, Number Theory (New York, 1984-1985).
Vol. \textbf{1240}. Lecture Notes in Math. Springer, 1987, pp. 275-280.

\bibitem{ThoFinite} J. G. Thompson, \emph{Finite groups and modular functions}, Bull. London Math. Soc. \textbf{11} (1979), no. 3, 347-351. MR MR554401 (81j:20029)

\bibitem{Tho} J. G. Thompson, \emph{Some numerology between the Fischer-Griess Monster and the elliptic modular function}, Bull. London Math. Soc. \textbf{11} (1979), no. 3, 352-353. MR MR554402 (81j:20030)

\bibitem{monster_subgroup}
R. A. Wilson, \emph{The odd-local subgroups of the Monster}, J. Austral. Math. Soc. Ser. A \textbf{44} (1988) 1-16.

\bibitem{wilson_communication}
R. A. Wilson, Personal communication, June 2017.

\bibitem{traces}
D. Zagier, \emph{Traces of singular moduli}, Motives, Polylogarithms and Hodge Theory, Lecture Series 3, International Press, Somerville (2002), 209-244. 

\bibitem{zhu}
Y. Zhu, \emph{Modular invariance of characters of vertex operator algebras}, J. Amer,Math. \emph{Soc.} \textbf{9} (1996), 237-302.


\end{thebibliography}
\end{document}